\documentclass[11pt]{amsart}
\usepackage{amsmath, amsthm, amscd, amsfonts,amssymb, graphicx}
\usepackage{accents}
\usepackage{algorithm} 
\usepackage{algpseudocode} 
\usepackage{xcolor}
\usepackage{hyperref}
\usepackage{appendix}
\usepackage{tikz-cd}
\usepackage{capt-of}

\newlength{\dhatheight}

\newcommand{\bea}{\begin{eqnarray*}}
\newcommand{\eea}{\end{eqnarray*}}

\newcommand{\beq}{\begin{equation}}
\newcommand{\eeq}{\end{equation}}
\newcommand{\bfomega}{\mbox{\boldmath $\omega$ \unboldmath} \hskip -0.05 true in}

\newcommand{\bom}{\bfomega}

 \newcommand{\D}{\mathrm{d}}
 \newcommand{\ii}{\mathrm{i}}

\newtheorem{theorem}{Theorem}[section]
\newtheorem{lemma}[theorem]{Lemma}
\theoremstyle{definition}

\newtheorem{example}[theorem]{Example}

\newtheorem{proposition}[theorem]{Proposition}
\newtheorem{corollary}[theorem]{Corollary}

\theoremstyle{remark}
\newtheorem{remark}[theorem]{Remark}

\numberwithin{equation}{section}

\begin{document}

\title[FFBT/iFFBT of convolutions using FFT]{Asymptotically steerable Finite Fourier-Bessel Transforms and Closure under Convolution}

\author[A. Ghaani Farashahi]{Arash Ghaani Farashahi$^{*,1}$}
\address{$^1$Department of Mechanical Engineering, National University of Singapore,
9 Engineering Drive 1, Singapore 117575.}
\email{arash.ghaanifarashahi@nus.edu.sg; ghaanifarashahi@outlook.com}

\author[G.S. Chirikjian]{Gregory S. Chirikjian$^{2}$}
\address{$^2$Department of Mechanical Engineering, National University of Singapore,
9 Engineering Drive 1, Singapore 117575.}
\email{mpegre@nus.edu.sg}


\thanks{Emails: arash.ghaanifarashahi@nus.edu.sg (Arash Ghaani Farashahi),  mpegre@nus.edu.sg (Gregory S. Chirikjian)}
\date{\today}

\begin{abstract}
This paper develops a constructive numerical scheme for Fourier-Bessel approximations on disks compatible with convolutions supported on disks.  We address accurate finite Fourier-Bessel transforms (FFBT) and inverse  finite Fourier-Bessel transforms (iFFBT) of functions on disks using the discrete Fourier Transform (DFT) on Cartesian grids.  Whereas the DFT and its fast implementation (FFT) are ubiquitous and are powerful for computing convolutions, they are not exactly steerable under rotations.  In contrast, Fourier-Bessel expansions are steerable, but lose both this property and the preservation of band limits under convolution. This work captures the best features of both as the band limit is allowed to increase. The convergence/error analysis and asymptotic steerability of FFBT/ iFFBT are investigated.  Conditions are established for the FFBT to converge to the Fourier-Bessel coefficient and for the iFFBT to uniformly approximate the Fourier-Bessel partial sums.  The matrix form of the finite transforms is discussed.  The implementation of the discrete method to compute numerical approximation of convolutions of compactly supported functions on disks is considered as well.  
\end{abstract}

\maketitle
\tableofcontents

\section{{\bf Introduction}}

Constructive approximation techniques for planar and spatial functions/signals with certain rotational symmetry and convolution (resp.  correlation) of such functions have been utilized  
in different applications such as applied mechanics \cite{wang},  image processing \cite{AR.MS.JA.AB, MarMicSinger, ZZ.YS.AS},  cryo-electron microscopy (cryo-EM) \cite{YC.NG.PAP.TW, EN.SHS,  ZZ.AS},    macromolecular docking \cite{19Venkatraman2009, 19Venkatraman2009a},  template matching techniques in digital image analysis for pattern recognition \cite{Kim.Kim},  Minkowski sum visualization \cite{Kavraki},  error propagation \cite{PI5},  computer vision \cite{GAP.YSB.DAK.BGM} and convolutional neural networks (CNNs) \cite{XC.QQ.RC.GS, TSC.MW, MW.FAH.MS}.

For the purpose of constructive approximation of 2D functions/signals with certain rotational symmetry,  it is canonically suggested to investigate the functions/signals in polar coordinates.  It would be analytically and practically attractive if the function can be decomposed into building blocks consisting of discrete orthonormal basis elements (which can also be viewed as basic wave-like patterns) having separable form in polar coordinates.  This implies that the associated decomposition is made
up of radial and angular decompositions and hence are steerable \cite{MarMicSinger}.

One of the natural suggestions for such a (discrete) polar orthonormal basis for the Hilbert space of square-integrable functions on disks,  is the Fourier-Bessel basis \cite{SLT.IA, AGHF.GSC.M, G.N.W, Zet.An}.   The Fourier-Bessel basis (normalized polar harmonics) elements on disks are  constructed by utilizing normalized Bessel functions in radial direction and normalized trigonometric monomials for rotational direction.  

Suppose $(m,n)\in\mathbb{Z}\times\mathbb{N}$.  The normalized polar harmonics $\Psi_{m,n}:\mathbb{R}^2\to\mathbb{C}$ is defined by
\begin{equation*}
\Psi_{m,n}(r,\theta):=\left\{\begin{array}{lll}\displaystyle\frac{e^{\ii m\theta}J_m(z_{m,n}r)}{\sqrt{\pi}|J_{m+1}(z_{m,n})|}, &\ {\rm if}\ r\le 1\\
0 &\ {\rm if}\ r>1
\end{array}\right.,
\end{equation*} 
where $J_m(x)$ is the Bessel function of the first kind of order $m$ and $\{z_{m,n}:n\in\mathbb{N}\}$ is the ordered set of all positive zeros of $J_m(x)$.  
Assume that $\mathbb{B}$ is the closed unit disk in $\mathbb{R}^2$.  Then every square integrable function $f:\mathbb{R}^2\to\mathbb{C}$ supported in $\mathbb{B}$,  can be approximated in the $L^2$-sense using Fourier-Bessel partial sums (also known as inverse Fourier-Bessel transforms) 
\[
S_{M,N}(f):=\sum_{m=-M}^M\sum_{n=1}^NC_{m,n}(f)\Psi_{m,n},
\]
with large enough $M,N\in\mathbb{N}$,  where the Fourier-Bessel coefficient/transform $C_{m,n}(f)$ are given by  
\[
C_{m,n}(f):=\int_0^{2\pi}\int_0^1f(r,\theta)\overline{\Psi_{m,n}(r,\theta)}r\D r \D\theta.
\]
In this case,  the steerability property reads as 
\[
C_{m,n}(R_\phi f)=e^{\ii m\phi}C_{m,n}(f),
\]
where $R_\phi f$ is the rotation of $f$ through the angle $\phi\in[0,2\pi)$, that is $R_\phi f(r,\theta):=f(r,\theta+\phi)$.

The analytical aspects of Fourier-Bessel transforms for functions on disks is discussed in \cite{AGHF.GSC.AA}.  It is shown that the Fourier-Bessel transform $C_{m,n}(f)$ of a square integrable function $f:\mathbb{R}^2\to\mathbb{C}$  supported in $\mathbb{B}$ with absolutely convergent 2D Fourier series satisfies the following closed form in terms of discrete sampling of 2D Fourier integrals
\begin{equation}\label{CmnXiSum}
C_{m,n}(f)=\sum_{\mathbf{k}\in\mathbb{Z}^2}\mathbf{c}(\mathbf{k};m,n)\widehat{f}(\mathbf{k}),
\end{equation}
where $\mathbf{c}(\mathbf{k};m,n)$ are unique complex numbers independent of $f$ and 
\[
\widehat{f}(\mathbf{k})=\int_0^{2\pi}\int_0^1f(r,\theta)e^{-\pi\ii r(k_1\cos\theta+k_2\sin\theta)}r\D r\D\theta,\hspace{1cm}{\rm for}\ \mathbf{k}=(k_1,k_2)^T\in\mathbb{Z}^2.
\]
There are several canonical advantageous for the constructive closed form (\ref{CmnXiSum}).  To begin with,  the infinite double sum in the right hand side of (\ref{CmnXiSum}) reduces into a finite sum for band-limited functions.  In addition,  due to the convolution property of 2D Fourier integrals, it can be naturally applied for convolution of arbitrary functions supported on disks.  

This paper investigates numerical aspects of Fourier-Bessel transforms and Fourier-Bessel series on disks using a computational scheme involving a proportional truncation of the series on the right hand side of (\ref{CmnXiSum}) and finite sampling of 2D Fourier integrals on disks.  We then utilize the discussed  computational strategy as a numerical approximation of convolutions supported on disks in terms of numerical Fourier integrals of convolved functions.   

Section 2 is devoted to review of general notation,  classical Fourier analysis, and Fourier-Bessel series on disks.  Section 3 introduces the notion of finite Fourier-Bessel transforms (FFBT) of any order $K\in\mathbb{N}$,  denoted by $C_{m,n}^{K}(f)$,  as an accurate numerical approximation of Fourier-Bessel transforms $C_{m,n}(f)$.  The basic  properties of the finite Fourier-Bessel transforms are then considered. 
In addition, convergence analysis of finite Fourier-Bessel transforms and asymptotic steerability is performed.  It is proved that if a function $f\in\mathcal{C}^2(\mathbb{R}^2)$ is supported in the open unit disk $\mathbb{B}^\circ$,  the finite Fourier-Bessel transform of order $K\in\mathbb{N}$ satisfies 
\[
\left|C_{m,n}(f)-C_{m,n}^{K}(f)\right|=\mathcal{O}(1/K),
\]
if $K\ge K_{m,n}:=\lceil\frac{z_{m,n}}{\pi}\rceil$.
Furthermore,  in this case the finite Fourier-Bessel transforms satisfy the following asymptotic steerability  
\[
\left|C_{m,n}^{K}(R_\phi f)-e^{\ii m\phi}C_{m,n}^{K}(f)\right|=\mathcal{O}(1/K),
\]
for every $\phi\in [0,2\pi)$.  Next, we study the inverse finite Fourier-Bessel transforms (iFFBT) of order $K\in\mathbb{N}$,  defined by 
\[
S_{M,N}^{K}(f):=\sum_{m=-M}^M\sum_{n=1}^NC_{m,n}^{K}(f)\Psi_{m,n}.
\]
We discuss conditions in which $S_{M,N}^{K}(f)$ can be considered as computable and accurate approximations for partial Fourier-Bessel sums $S_{M,N}(f)$.  It is shown that the inverse finite Fourier-Bessel transform of order $K\in\mathbb{N}$ of the partial Fourier-Bessel sum $S_{M,N}(f)$ of every $f\in\mathcal{C}^2(\mathbb{R}^2)$ supported in $\mathbb{B}^\circ$,  satisfies 
\[
\left\|S_{M,N}(f)-S_{M,N}^{K}(f)\right\|_\infty=\mathcal{O}(1/K),
\]
if $K\ge K[M,N]:=\max\{K_{m,n}:0\le|m|\le M,1\le n\le N\}$.
Section 4 presents matrix forms for FFBT/iFFBT and then illustrates some numerical experiments for FFBT/iFFBT of functions on disks in MATLAB.  It is shown that FFBT/iFFBT satisfy closed matrix forms in terms of discrete Fourier transforms,  matrix multiplications and the matrix trace functional which can be implemented using fast Fourier (resp.   matrix multiplication) algorithms.  

Section 5 develops the computational scheme for approximating convolution of  functions on disks.  We discuss a unified version of finite Fourier-Bessel transforms and the associated inverse finite Fourier-Bessel transforms for convolution of compactly supported functions on disks.  Assume that functions $f,g:\mathbb{R}^2\to\mathbb{C}$ are supported in $\frac{1}{2}\mathbb{B}$. We introduce the unified finite Fourier-Bessel transform of order $K\in\mathbb{N}$ associated to $f,g$,  also denoted by $C_{m,n}^{K}[f,g]$,  as an accurate numerical approximation of $C_{m,n}^{K}(f\ast g)$.  It is shown that
if $f,g\in\mathcal{C}^1(\mathbb{R}^2)$ then 
\[
\left|C_{m,n}^{K}[f,g]-C_{m,n}^{K}(f\ast g)\right|=\mathcal{O}(1/K),
\] 
In addition, it is proven that the unified inverse finite Fourier-Bessel transforms for convolutions,  given by 
\[
S_{M,N}^{K}[f,g]:=\sum_{m=-M}^M\sum_{n=1}^NC_{m,n}^{K}[f,g]\Psi_{m,n},
\]
satisfies 
\[
\left\|S_{M,N}^{K}[f,g]-S_{M,N}(f\ast g)\right\|_\infty=\mathcal{O}(1/K),
\]
if $K\ge K[M,N]$.  The main advantage of the unified finite Fourier-Bessel transforms of convolutions is that it can be applied to convolution of any compactly supported functions $f,g$ without any samplings from the values of the convolution integral $f\ast g$.  Section 6 addresses closed matrix forms of the unified FFBT/iFFBT for convolutions together with  some numerical convolution experiments in MATLAB.  
\section{{\bf Preliminaries and Notation}}
Throughout this section we fix notations and review some preliminaries.
\subsection{General notation} Suppose $a>0$, $q\in\mathbb{N}$ and $p\in (0,\infty]$.   
\begin{enumerate}
\item For sequences $(a_n),(b_n)\in\mathbb{C}^\mathbb{N}$ and $N\in\mathbb{N}$,  we write $a_n=\mathcal{O}(b_n)$ for $n\ge N$, if there exists $M>0$ such that $|a_n|\le M|b_n|$ for $n\ge N$.
\item If $\mathbf{x}\in\mathbb{R}^2$ and $p<\infty$,  $\|\mathbf{x}\|_p^p$ denotes $\sum_{i=1}^2|x_i|^p$
\item If $\mathbf{x}\in\mathbb{R}^2$,  $\|\mathbf{x}\|_{\infty}$ denotes $\max\{|x_i|:1\le i\le 2\}$
\item If $\mathbf{x},\mathbf{y}\in\mathbb{R}^2$,  $\langle\mathbf{x},\mathbf{y}\rangle:=x_1y_1+x_2y_2$ denotes the dot product (standard inner product) of $\mathbb{R}^2$
\item $\Omega_a$ denotes the rectangle $[-a,a]\times [-a,a]$ in $\mathbb{R}^2$ and $\Omega$ denotes $\Omega_1$.
\item $\mathbb{B}_a=\{\mathbf{x}\in\mathbb{R}^2:\|\mathbf{x}\|_2\le a\}$ denotes the closed disk of radius $a$ in $\mathbb{R}^2$ and $\mathbb{B}$ denotes $\mathbb{B}_1$.
\item $\mathbb{B}^\circ_a=\{\mathbf{x}\in\mathbb{R}^2:\|\mathbf{x}\|_2<a\}$ denotes the open disk of radius $a$ in $\mathbb{R}^2$ and $\mathbb{B}^\circ$ denotes $\mathbb{B}_1^\circ$.
\item $u,v,\ldots$ denote complex-valued functions on $[-1,1]$
\item $U,V,\ldots$ denote complex-valued functions on $\Omega$
\item $f,g,\ldots$ denote functions on $\mathbb{R}^2$.
\item For a function $f:\mathbb{R}^2\to\mathbb{C}$, $\mathrm{supp}(f):=\{\mathbf{x}\in\mathbb{R}^2:f(\mathbf{x})\not=0\}$ denotes the support of $f$
\item $\mathcal{V}_a$ denotes the set of all square integrable functions on $\mathbb{R}^2$ supported in $\mathbb{B}_a$, i.e.  $\mathrm{supp}(f)\subseteq\mathbb{B}_a$ and $\mathcal{V}$ denotes $\mathcal{V}_1$.
\item For $f\in\mathcal{V}_a$,  $\|f\|_\infty$ denotes $\sup_{\mathbf{x}}|f(\mathbf{x})|$.
\item $\mathcal{C}^q(X)$ denotes the set of all $f:X\subseteq\mathbb{R}^2\to\mathbb{C}$ such that all partial derivatives of order $\le q$ exist and are continuous.
\item For $f\in\mathcal{C}^1(X)$,  $\nabla f:X\subseteq\mathbb{R}^2\to\mathbb{C}$ denotes the gradient of $f$.
\item For $X\subseteq\mathbb{R}^2$ and $1\le p<\infty$,  $L^p(X)$ denotes the space of measurable functions $f:X\to\mathbb{C}$ for which the $p$-th power of the absolute value is Lebesgue integrable.
\end{enumerate}
\subsection{Fourier analysis}
The Lebesgue integration on $\mathbb{R}^2$ satisfies the following polar decomposition 
\begin{equation*}\label{int.polar.dec}
\int_{\mathbb{R}^2}f(\mathbf{x})\D\mathbf{x}=\int_{0}^{2\pi}\int_0^\infty f(r,\theta)r\D r\D\theta,
\end{equation*}
for every $f\in L^1(\mathbb{R}^2)$, see Theorem 2.49 of \cite{Foll.R}.  If $f$ is supported in $\mathbb{B}_a$, then 
\begin{equation}\label{int.polar.dec.a}
\int_{\mathbb{R}^2}f(\mathbf{x})\D\mathbf{x}=\int_{0}^{2\pi}\int_0^a f(r,\theta)r\D r\D\theta.
\end{equation}

For $f,g\in L^1(\mathbb{R}^2)$,  the convolution integral $f\ast g\in L^1(\mathbb{R}^2)$ at $\mathbf{x}\in\mathbb{R}^2$,  is defined by \cite{Foll.R}
\begin{equation}\label{conv.Rd}
(f\ast g)(\mathbf{x}):=\int_{\mathbb{R}^2}f(\mathbf{y})g(\mathbf{x}-\mathbf{y})\D\mathbf{y}.
\end{equation}
Also,  if $a,b>0$ and $f,g\in L^1(\mathbb{R}^2)$ with $\mathrm{supp}(f)\subseteq\mathbb{B}_{a}$ (resp.  $\mathrm{supp}(g)\subseteq\mathbb{B}_{b}$) then $\mathrm{supp}(f\ast g)\subseteq\mathbb{B}_{a+b}$.

The Fourier integral of $f\in L^1(\mathbb{R}^2)$ at $\bom\in\mathbb{R}^2$, is defined by
\begin{equation}\label{Ff}
\widehat{f}(\bom):=\int_{\mathbb{R}^2}f(\mathbf{x})e^{-\ii\pi\langle\bom,\mathbf{x}\rangle}\D\mathbf{x}.
\end{equation}
If $f,g\in L^1(\mathbb{R}^2)$ and $\bom\in\mathbb{R}^2$, the Fourier integral (\ref{Ff}) satisfies the following convolution property 
\begin{equation}\label{Ff*g}
\widehat{f\ast g}(\bom)=\widehat{f}(\bom)\widehat{g}(\bom).
\end{equation}

Suppose $L\in\mathbb{N}$ and $\mathbb{C}^{L\times L}$ is the linear space of all square matrices with complex entries of size $L$. The discrete Fourier transform (DFT) of $\mathbf{A}\in\mathbb{C}^{L\times L}$, is defined as the matrix $\widehat{\mathbf{A}}\in\mathbb{C}^{L\times L}$ given by 
\begin{equation}\label{DFT}
\widehat{\mathbf{A}}(l,\ell):=\sum_{k=1}^L\sum_{j=1}^L\mathbf{A}(k,j)e^{-2\pi\ii(k-1)(l-1)/L}e^{-2\pi\ii(j-1)(\ell-1)/L},
\end{equation}
for every $1\le l,\ell\le L$.  
It is well known that the DFT can be computed efficiently using the Fast Fourier Transform (FFT) algorithm, see \cite{Av.Co.Do.El.Is, Fe.Ku.Po, St} and references therein.

\subsection{Fourier-Bessel series on disks}
Suppose $m\in\mathbb{Z}$ and $J_m(x)$ is the Bessel function of the first kind of order $m$.  Let $\{z_{m,n}:n\in\mathbb{N}\}$ be the set of all positive zeros of $J_m(x)$ 
such that $z_{m,1}<z_{m,2}<\cdots<z_{m,n}<\cdots$, i.e. $J_m(z_{m,n})=0$ for every $n\in\mathbb{N}$. 
For $n\in\mathbb{N}$,  the normalized Bessel function $\mathcal{J}_{m,n}$ on $[0,1]$ is given by
\begin{equation}\label{Jnm.a}
\mathcal{J}_{m,n}(r):=\frac{\sqrt{2}}{|J_{m+1}(z_{m,n})|}J_m(z_{m,n}r),\hspace{1cm}\ {\rm for}\ r\in[0,1].
\end{equation}
Then 
\[
\int_0^1\mathcal{J}_{m,n}(r)\mathcal{J}_{m,n'}(r)r\D r=\delta_{n,n'}
\]
Therefore,  any $\nu\in L^2([0,1],r\D r)$, 
satisfies the following constructive $L^2$-expansion  \cite{SLT.IA,  G.N.W, Zet.An}
\begin{equation}\label{mFBS}
\nu=\sum_{n=1}^\infty \langle \nu,\mathcal{J}_{m,n}\rangle\mathcal{J}_{m,n},
\hspace{1cm}{\rm where}\hspace{0.5cm} \langle \nu,\mathcal{J}_{m,n}\rangle:=\int_0^1\nu(s)\mathcal{J}_{m,n}(s)s\D s,
\end{equation}
for every $n\in\mathbb{N}$. 
It is worthwhile to mentioned that the series (\ref{mFBS}) converges in the mean; that is,
\[
\lim_{N\to\infty}\int_0^1\left|\nu(r)-\sum_{n=1}^N  \langle\nu,\mathcal{J}_{m,n}\rangle\mathcal{J}_{m,n}(r)\right|^2r\D r=0.
\]
\begin{figure}[H]
\centering
\includegraphics[keepaspectratio=true,width=\textwidth, height=0.2\textheight]{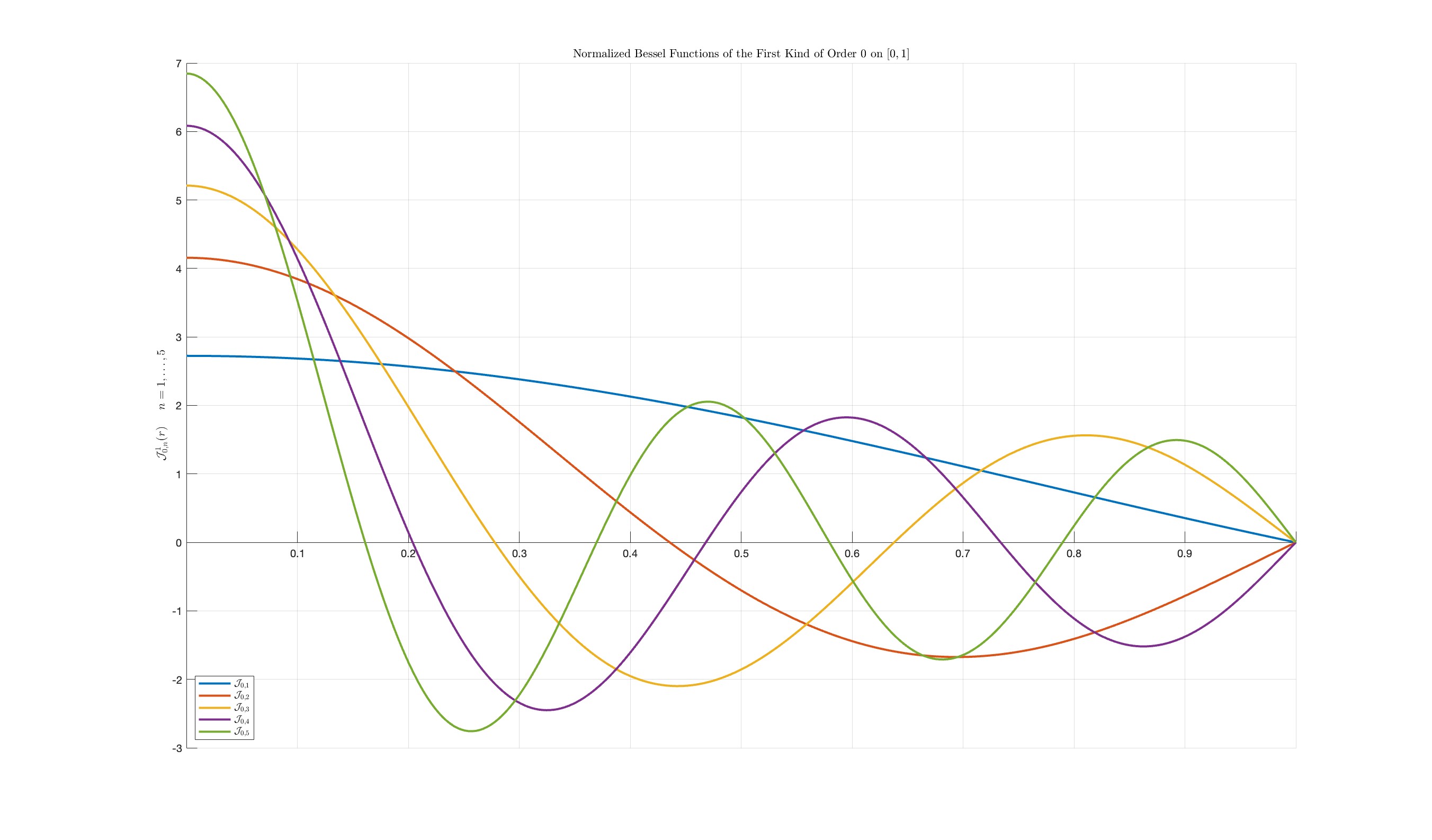}
\caption{Normalized Bessel functions of the first kind of order $0$ on $[0,1]$. }
\label{fig:J05}
\end{figure}
\begin{figure}[H]
\centering
\includegraphics[keepaspectratio=true,width=\textwidth, height=0.2\textheight]{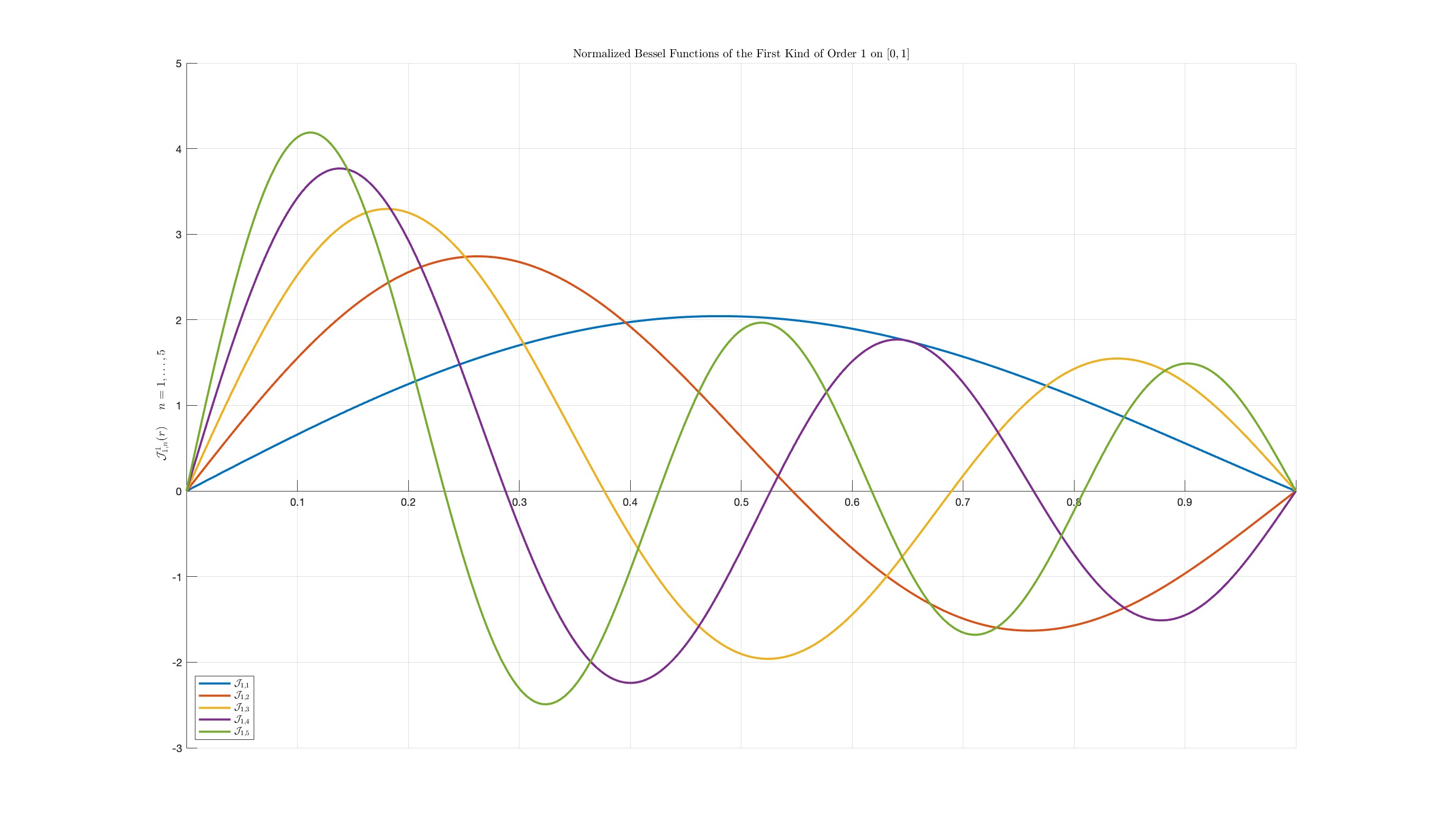}
\caption{Normalized Bessel functions of the first kind of order $1$ on $[0,1]$. }
\label{fig:J15}
\end{figure}
The normalized polar harmonics on $\mathbb{R}^2$ are defined by
\begin{equation}\label{2D.B.Fr.gp}
\Psi_{m,n}(r\mathbf{u}_\theta)=\Psi_{m,n}(r,\theta):=\left\{\begin{array}{lll}\frac{1}{\sqrt{2\pi}}\exp(\ii m\theta)\mathcal{J}_{m,n}(r)&\ {\rm if}\ r\le 1\\
0 &\ {\rm if}\ r>1
\end{array}\right.,
\end{equation}
where the normalized radial function $\mathcal{J}_{m,n}(r)$ is defined by (\ref{Jnm.a}).   We then have 
\begin{equation*}
\int_0^{2\pi}\int_0^1\Psi_{m,n}(r,\theta)\overline{\Psi_{m',n'}(r,\theta)}r\D r\D\theta=\delta_{n,n'}\delta_{m,m'}.
\end{equation*}
Every $f\in\mathcal{V}$ can be expanded in $L^2$-sense with respect to basis elements $\Psi_{m,n}$ given in (\ref{2D.B.Fr.gp}) by
\begin{equation}\label{FBSp}
f=\sum_{m=-\infty}^\infty\sum_{n=1}^{\infty} C_{m,n}(f)\Psi_{m,n},
\end{equation}
where
\begin{equation}\label{FBT}
C_{m,n}(f):=\int_0^{2\pi}\int_0^1f(r,\theta)\overline{\Psi_{m,n}(r,\theta)}r\D r \D\theta.
\end{equation}
The convergence of the series (\ref{FBSp}) is interpreted in the mean; that is,
\[
\lim_{M\to\infty}\lim_{N\to\infty}\int_0^{2\pi}\int_0^1\left|f(r,\theta)-S_{M,N}(f)(r,\theta)\right|^2r\D r\D\theta=0,
\]
where 
\begin{equation}\label{SMN}
S_{M,N}(f)(r,\theta):=\sum_{m=-M}^M\sum_{n=1}^NC_{m,n}(f)\Psi_{m,n}(r,\theta),
\end{equation}
for every $r\ge 0$ and $\theta\in[0,2\pi)$.
Suppose that $\mathcal{A}$ is the subspace of $\mathcal{V}$ given by 
\[
\mathcal{A}:=\left\{f\in\mathcal{V}:\sum_{\mathbf{k}\in\mathbb{Z}^2}\left|\widehat{f}(\mathbf{k})\right|<\infty\right\},
\] 
where (\ref{int.polar.dec.a}) and (\ref{Ff}) imply that 
\begin{equation}\label{C.ZBVC.alt}
\widehat{f}(\mathbf{k})=\int_0^{2\pi}\int_0^1f(r,\theta)e^{-\pi\ii r(k_1\cos\theta+k_2\sin\theta)}r\D r\D\theta,
\end{equation}
for every $\mathbf{k}:=(k_1,k_2)^T\in\mathbb{Z}^2$. 
Theorem 3.3 of \cite{AGHF.GSC.AA} showed that if $f\in\mathcal{A}$ and $(m,n)\in\mathbb{Z}\times\mathbb{N}$ then 
\begin{equation}\label{C.ZBVC}
C_{m,n}(f)=\sum_{\mathbf{k}\in\mathbb{Z}^2}\mathbf{c}(\mathbf{k};m,n)\widehat{f}(\mathbf{k}),
\end{equation}
where 
\begin{equation}\label{cakmn}
\mathbf{c}(\mathbf{k};m,n):=\left\{\begin{array}{lll}
\sqrt{\pi}(-1)^{n}\ii^{m}z_{m,n}\frac{J_{m}(\pi\|\mathbf{k}\|_2)e^{-\ii m\Phi(\mathbf{k})}}{2(\pi^2\|\mathbf{k}\|_2^2-z_{m,n}^2)} & {\rm if}\ m\ge 0\\
(-1)^m\sqrt{\pi}(-1)^{n}\ii^{m}z_{m,n}\frac{J_{m}(\pi\|\mathbf{k}\|_2)e^{-\ii m\Phi(\mathbf{k})}}{2(\pi^2\|\mathbf{k}\|_2^2-z_{m,n}^2)} & {\rm if}\ m< 0\\
\end{array}\right.,
\end{equation}
for every $\mathbf{k}=(k_1,k_2)^T\in\mathbb{Z}^2$, where $\Phi(\mathbf{k}):=\mathrm{atan2}(k_2,k_1)$.
\section{\bf{Finite Fourier-Bessel Transforms}}
Throughout,  we introduce the notion of finite Fourier-Bessel transforms (FFBT),  as numerical approximations of Fourier-Bessel coefficients given by (\ref{FBT}). Then using the notion of FFBT, we develop inverse finite Fourier-Bessel transform (iFFBT) as numerical approximations of Fourier-Bessel partial sums given by (\ref{SMN}). 
We also study basic properties of these finite  transforms.  The motivation for introducing these finite transforms originated from the closed form (\ref{C.ZBVC}).  

To begin with, we discuss a numerical integration scheme for Fourier integrals of the form (\ref{C.ZBVC.alt}) using finite Fourier transforms on rectangles which can be reformulated in closed form in terms of DFT and hence can be implemented by FFT,  as discussed in Section \ref{MatFFT}.  

\subsection{Finite Fourier transforms on disks}\label{FFtD}

This part studies a unified computational approach to compute/approximate the Fourier integrals of the form $\widehat{f}(\mathbf{k})$, where $f:\mathbb{R}^2\to\mathbb{C}$ is a function supported in $\mathbb{B}$ and $\mathbf{k}\in\mathbb{Z}^2$,  using uniform square sampling on disks and finite Fourier transforms which they can be implemented via fast Fourier algorithms,  see Section \ref{MatFFT}. 

Suppose $f\in\mathcal{V}$,  $L\in\mathbb{N}$ and $\bom:=(\omega_1,\omega_2)^T\in\mathbb{R}^2$.
The finite Fourier transform of $f$ at $\bom$ on the uniform square grid of size $L\times L$ is defined by   
\begin{equation}\label{Xi.kL}
\widehat{f}(\bom;L):=\delta^2\sum_{i=1}^L\sum_{j=1}^Lf(x_i,x_j)e^{-\pi\ii (\omega_1x_i+\omega_2x_j)},
\end{equation}
where $x_i:=-1+(i-1)\delta$ with $\delta:=\frac{2}{L}$, for $1\le i\le L$.  
Since $f$ is supported in $\mathbb{B}$,  we get
\begin{equation*}
\widehat{f}(\bom;L)=\delta^2\sum_{(i,j)\in\mathbb{B}^L}f(x_i,x_j)e^{-\pi\ii(\omega_1x_i+\omega_2x_j)},
\end{equation*}
for every $\bom:=(\omega_1,\omega_2)\in\mathbb{R}^2$, where $\mathbb{N}_L:=\{i\in\mathbb{N}:1\le i\le L\}$ and 
\[
\mathbb{B}^L:=\left\{(i,j)\in\mathbb{N}_L\times\mathbb{N}_L:(x_i,x_j)\in\mathbb{B}\right\}.
\]
\begin{figure}[H]
\centering
\includegraphics[keepaspectratio=true,width=\textwidth, height=0.3\textheight]{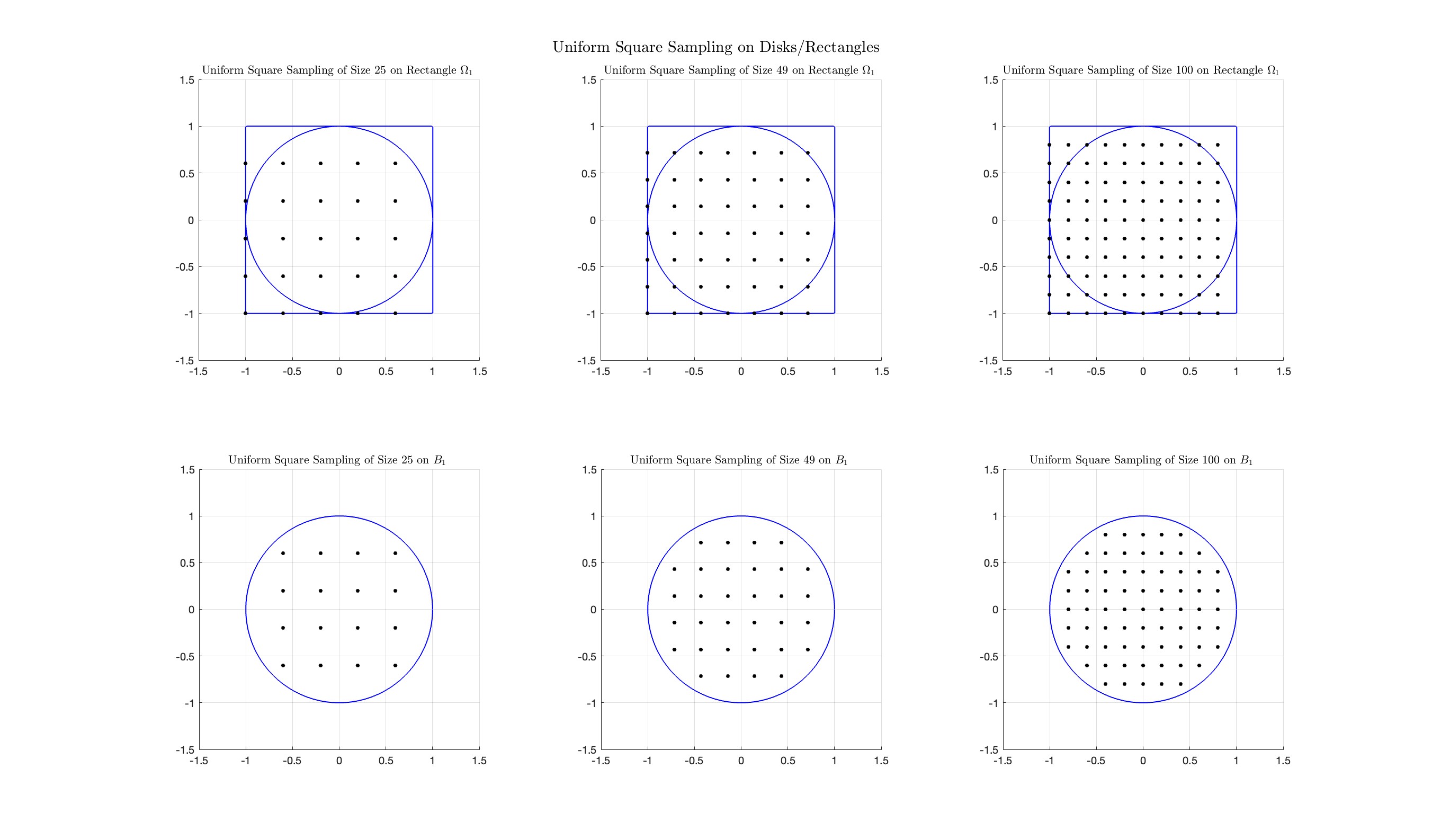}
\caption{Sampling points $(x_i,x_j)$ with $(i,j)\in \mathbb{B}^L$ plotted in the square $\Omega$ and disk $\mathbb{B}$ for (left) $L = 5$,  (middle) $L=7$, and (right) $L =10$.}
\label{fig:SamPoints}
\end{figure}

The following proposition presents some symmetry properties of finite Fourier transforms on disks. 

\begin{proposition}\label{FFTmainSym}
{\it Let $f\in\mathcal{V}$ and $L\in\mathbb{N}$.
Suppose that $x_j:=-1+(j-1)\delta$ with $\delta:=\frac{2}{L}$ and $1\le j\le L$.  Assume $\bom=(\omega_1,\omega_2)^T\in\mathbb{R}^2$. Then
\begin{enumerate}
\item $\overline{\widehat{f}(\bom;L)}=\widehat{\overline{f}}(-\bom;L)$.
\item $\overline{\widehat{f}(\bom;L)}=\widehat{f}(-\bom;L)$, if $f$ is real-valued.
\end{enumerate}
}\end{proposition}
\begin{proof}
(1) Invoking (\ref{Xi.kL}), we have 
\begin{align*}
\overline{\widehat{f}(\bom;L)}=\delta^2\sum_{i=1}^L\sum_{j=1}^L\overline{f(x_i,x_j)}e^{\ii (\omega_1x_i+\omega_2x_j)}=\widehat{\overline{f}}(-\bom;L).
\end{align*}
(2) follows from (1).\\
\end{proof}

The following proposition investigates convergence rate of finite Fourier transforms of the form $\widehat{f}(\mathbf{k};L)$ as an accurate numerical approximation of the Fourier integrals $\widehat{f}(\mathbf{k})$. To this end, we discussed some related technicalities in Appendix \ref{Apx1}. 

\begin{proposition}\label{mainXIb}
{\it Let $f\in \mathcal{C}^1(\mathbb{R}^2)$ be a function supported in  $\mathbb{B}^\circ$.  Suppose $K\in\mathbb{N}$ and $\mathbf{k}\in\mathbb{Z}^2$.  Then  
\[
\left|\widehat{f}(\mathbf{k})-\widehat{f}(\mathbf{k};2K+1)\right|\le\frac{96\|\nabla f\|_\infty}{\pi K},\hspace{1cm}{\rm if}\ \|\mathbf{k}\|_\infty\le K.
\]
}\end{proposition}
\begin{proof}
Let $U$ be the restriction of $f$ to the rectangle $\Omega$.  Then $U$ is periodic and $U\in \mathcal{C}^1(\Omega)$.  Then 
\[
\widehat{f}(\mathbf{k})=4\widehat{U}[\mathbf{k}], \hspace{1cm}
\widehat{f}(\mathbf{k};L)=4\widehat{U}[\mathbf{k};2K+1].
\]
Therefore,  using Proposition \ref{mainT},  we get 
\begin{align*}
\left|\widehat{f}(\mathbf{k})-\widehat{f}(\mathbf{k};2K+1)\right|
=4\left|\widehat{U}[\mathbf{k}]-\widehat{U}[\mathbf{k};2K+1]\right|\le\frac{96\|\nabla f\|_\infty}{\pi K}.
\end{align*}
\end{proof}
\begin{remark}
Proposition \ref{mainXIb} reads as the following accuracy result. 
Suppose $f\in \mathcal{C}^1(\mathbb{R}^2)$ is supported in $\mathbb{B}^\circ$.  Let $\mathbf{k}\in\mathbb{Z}^2$ and $\epsilon>0$.  Assume $K:=\lceil \max\{\epsilon^{-1}\pi^{-1}96\|\nabla f\|_\infty,\|\mathbf{k}\|_\infty\}\rceil$ and $L:=2K+1$. Then $\widehat{f}(\mathbf{k};L)$ approximates the Fourier integral $\widehat{f}(\mathbf{k})$ with the absolute error less than $\epsilon$. 
\end{remark}

\subsection{Finite Fourier-Bessel transforms (FFBT)}\label{FFBT}
Here we introduce the notion of finite Fourier-Bessel transforms for functions supported on disks and will discuss fundamental properties of  finite Fourier-Bessel transforms.   It is shown that finite Fourier-Bessel transforms are converging to actual value of Fourier-Bessel coefficients.  The matrix form of FFBT discussed in \ref{MatFFBT} as well.

Let $f\in\mathcal{V}$ and $(m,n)\in\mathbb{Z}\times\mathbb{N}$.
The finite Fourier-Bessel transform ($\mathbf{FFBT}$) of order $K\in\mathbb{N}$ of $C_{m,n}(f)$ is given by 
\begin{equation}\label{DFBT-FFT}
C_{m,n}^K(f):=\sum_{\|\mathbf{k}\|_{\infty}\le K}\mathbf{c}(\mathbf{k};m,n)\widehat{f}(\mathbf{k};2K+1),
\end{equation}
with $\mathbf{c}(\mathbf{k};m,n)$ given by (\ref{cakmn}) for $\mathbf{k}:=(k_1,k_2)^T\in\mathbb{Z}^2$, and 
\[
\widehat{f}(\mathbf{k};2K+1):=\delta^2\sum_{i=1}^{2K+1}\sum_{j=1}^{2K+1}f(x_i,x_j)e^{-\pi\ii (k_1x_i+k_2x_j)},
\]
where $x_j:=-1+(j-1)\delta$ with $\delta:=\frac{2}{2K+1}$ for $1\le j\le 2K+1$.  

\begin{proposition}\label{DFBT-FFTmain}
{\it Let $f\in\mathcal{V}$ and $K,n\in\mathbb{N}$.
Suppose that $m\in\mathbb{Z}$ with $m\ge 0$. Then
\begin{enumerate}
\item $C_{-m,n}^{K}(f)=(-1)^m\overline{C_{m,n}^{K}(\overline{f})}$.
\item  $C_{-m,n}^{K}(f)=(-1)^m\overline{C_{m,n}^{K}(f)}$ if $f$ is real-valued.
\end{enumerate}
}\end{proposition}
\begin{proof}
(1) Let $L:=2K+1$. Using Proposition \ref{FFTmainSym} and Lemma \ref{cakmnSym}, we get 
\begin{align*}
\overline{C_{-m,n}^{K}(f)}
&=\sum_{\|\mathbf{k}\|_{\infty}\le K}\overline{\mathbf{c}(\mathbf{k};-m,n)}\overline{\widehat{f}(\mathbf{k};L)}
=\sum_{\|\mathbf{k}\|_{\infty}\le K}\mathbf{c}(\mathbf{k};m,n)\overline{\widehat{f}(\mathbf{k};L)}
=\sum_{\|\mathbf{k}\|_{\infty}\le K}\mathbf{c}(\mathbf{k};m,n)\widehat{\overline{f}}(-\mathbf{k};L)
\\&=\sum_{\|\mathbf{k}\|_{\infty}\le K}\mathbf{c}(-\mathbf{k};m,n)\widehat{\overline{f}}(\mathbf{k};L)
=(-1)^m\sum_{\|\mathbf{k}\|_{\infty}\le K}\mathbf{c}(\mathbf{k};m,n)\widehat{\overline{f}}(\mathbf{k};L)
=(-1)^mC_{m,n}^{K}(\overline{f}).
\end{align*}
(2) is straightforward from (1).
\end{proof}
We continue by investigation of convergence rate of finite Fourier-Bessel transform using finite Fourier transforms.  Some related technical details discussed in Appendix \ref{Apx}.  The following theorem shows that $C_{m,n}^{K}(f)$ is a numerical approximation of the Fourier-Bessel coefficient $C_{m,n}(f)$ which is also accurate if $K\ge K_{m,n}$ with $K_{m,n}:=\lceil\frac{z_{m,n}}{\pi}\rceil$.
\begin{theorem}\label{CmnKerr}
Let $f\in\mathcal{C}^1(\mathbb{R}^2)\cap\mathcal{A}$ be supported in $\mathbb{B}^\circ$ and $(m,n)\in\mathbb{Z}\times\mathbb{N}$.  Then there exist constants $\gamma_{m,n},c_f>0$ such that 
\[
\left|C_{m,n}(f)-C_{m,n}^{K}(f)\right|\le\frac{2\gamma_{m,n}c_f}{K},\hspace{1cm}{\rm for}\ K\ge K_{m,n}.
\]
\end{theorem}
\begin{proof}
Let $\mathbf{c}(\mathbf{k}):=\mathbf{c}(\mathbf{k};m,n)$.  Lemma \ref{SumJmk},  implies $\beta_{m,n}:=\sum_{\mathbf{k}\in\mathbb{Z}^2}|\mathbf{c}(\mathbf{k})|<\infty$.  So Proposition \ref{mainXIb} gives  
\begin{align*}
&\left|C_{m,n}^{K}(f)-C_{m,n}(f)\right|
=\left|\sum_{\|\mathbf{k}\|_\infty\le K}\mathbf{c}(\mathbf{k})\widehat{f}(\mathbf{k};2K+1)-\sum_{\mathbf{k}\in\mathbb{Z}^2}\mathbf{c}(\mathbf{k})\widehat{f}(\mathbf{k})\right|
\\&\le\sum_{\|\mathbf{k}\|_\infty\le K}|\mathbf{c}(\mathbf{k})||\widehat{f}(\mathbf{k};2K+1)-\widehat{f}(\mathbf{k})|+\sum_{\|\mathbf{k}\|_\infty>K}|\mathbf{c}(\mathbf{k})||\widehat{f}(\mathbf{k})|
\\&\le\frac{96\|\nabla f\|_\infty}{\pi K}\left(\sum_{\|\mathbf{k}\|_\infty\le K}|\mathbf{c}(\mathbf{k})|\right)+\sum_{\|\mathbf{k}\|_\infty>K}|\mathbf{c}(\mathbf{k})||\widehat{f}(\mathbf{k})|
\le\frac{96\|\nabla f\|_\infty}{\pi K}\beta_{m,n}+\sum_{\|\mathbf{k}\|_\infty>K}|\mathbf{c}(\mathbf{k})||\widehat{f}(\mathbf{k})|.
\end{align*}
Let $K\ge K_{m,n}$ and $\mathbf{k}\in\mathbb{Z}^2$ with $\|\mathbf{k}\|_\infty>K$.  Then $\pi^2\|\mathbf{k}\|_2^2>\pi^2K^2>z_{m,n}^2$,
implying that 
\begin{equation}\label{k2Kzmn}
\frac{1}{\pi^2\|\mathbf{k}\|_2^2-z_{m,n}^2}<\frac{1}{\pi^2K^2-z_{m,n}^2}.
\end{equation}
Hence,  if $\|\mathbf{k}\|_\infty>K\ge K_{m,n}$ then (\ref{k2Kzmn}) and Lemma \ref{OmainKmn}, for $t=K$ one can get
\begin{align*}
|\mathbf{c}(\mathbf{k})|
&=\sqrt{\pi}z_{m,n}\frac{|J_m(\pi\|\mathbf{k}\|_2)|}{|\pi^2\|\mathbf{k}\|_2^2-z_{m,n}^2|}=\sqrt{\pi}z_{m,n}\frac{|J_m(\pi\|\mathbf{k}\|_2)|}{\pi^2\|\mathbf{k}\|_2^2-z_{m,n}^2}
\\&\le\frac{ \sqrt{\pi}z_{m,n}}{\pi^2\|\mathbf{k}\|_2^2-z_{m,n}^2}<\frac{ \sqrt{\pi}z_{m,n}}{\pi^2K^2-z_{m,n}^2}\le\frac{\sqrt{\pi}K_{m,n}^2z_{m,n}}{\pi^2K_{m,n}^2-z_{m,n}^2}\frac{1}{K^2}=:\frac{\alpha_{m,n}}{K^2}.
\end{align*}
Therefore, if $K\ge K_{m,n}$ then we have  
\begin{align*}
\sum_{\|\mathbf{k}\|_\infty>K}|\mathbf{c}(\mathbf{k})||\widehat{f}(\mathbf{k})|
\le\frac{\alpha_{m,n}}{K^2}\sum_{\|\mathbf{k}\|_\infty>K}|\widehat{f}(\mathbf{k})|
\le \frac{\alpha_{m,n}}{K^2}\sum_{\mathbf{k}\in\mathbb{Z}^2}|\widehat{f}(\mathbf{k})|
=\frac{\alpha_{m,n}\|f\|_{\mathcal{A}}}{K^2},
\end{align*}
where $\|f\|_{\mathcal{A}}:=\sum_{\mathbf{k}\in\mathbb{Z}^2}|\widehat{f}(\mathbf{k})|$.  Henceforth, we get 
\begin{align*}
\left|C_{m,n}^{K}(f)-C_{m,n}(f)\right|&
\le\frac{96\|\nabla f\|_\infty\beta_{m,n}}{\pi K}+\frac{\alpha_{m,n}\|f\|_{\mathcal{A}}}{K^2}\le\frac{2c_f\gamma_{m,n}}{K},
\end{align*}
where $c_f:=\max\{\pi^{-1}96\|\nabla f\|_\infty,\|f\|_{\mathcal{A}}\}$ and $\gamma_{m,n}:=\max\{\alpha_{m,n},\beta_{m,n}\}$.
\end{proof}
\begin{remark}
Suppose that $(m,n)\in\mathbb{Z}\times\mathbb{N}$. In terms of approximation theory,   Theorem \ref{CmnKerr} reads as follows. If $f\in\mathcal{C}^1(\mathbb{R}^2)\cap\mathcal{A}$ is supported in $\mathbb{B}^\circ$ then   
$\left|C_{m,n}(f)-C_{m,n}^{K}(f)\right|=\mathcal{O}(1/K)$,
for every $K_{m,n}\le K$.  In details, if $\epsilon>0$ is given.  Assume $K:=\lceil \max\{\epsilon^{-1}2c_f\gamma_{m,n},K_{m,n}\}\rceil$. Then $C_{m,n}^K(f)$ approximates the Fourier-Bessel coefficient $C_{m,n}(f)$ with the absolute error less than $\epsilon$. 
\end{remark}
\begin{corollary}
{\it Let $f\in\mathcal{C}^2(\mathbb{R}^2)$ be supported in $\mathbb{B}^\circ$.  Suppose $(m,n)\in\mathbb{Z}\times\mathbb{N}$. Then 
\[
\left|C_{m,n}(f)-C_{m,n}^{K}(f)\right|=\mathcal{O}(1/K),\hspace{1cm}{\rm for}\ K\ge K_{m,n}.
\]
}\end{corollary}
\begin{remark}\label{FFBTstp}
The FFBT is compatible with steerable property of Fourier-Bessel basis in the asymptotic sense.  Let $f\in\mathcal{C}^2(\mathbb{R}^2)$ be supported in $\mathbb{B}^\circ$.  Suppose $(m,n)\in\mathbb{Z}\times\mathbb{N}$ and $\phi\in[0,2\pi)$.  Assume that $R_\phi f(r,\theta):=f(r,\theta+\phi)$ for every $(r,\theta)\in\mathbb{R}^2$.  Then 
$$\left|C_{m,n}^{K}(R_\phi f)-e^{\ii m\phi}C_{m,n}^{K}(f)\right|=\mathcal{O}(1/K) \hspace{1cm}{\rm for}\ K\ge K_{m,n}.$$
\end{remark}
The following diagram compares construction of Fourier-Bessel transform $C_{m,n}(f)$ and finite Fourier-Bessel transform $C_{m,n}^K(f)$ for $f\in\mathcal{C}^1(\mathbb{R}^2)\cap\mathcal{A}$ at order $K\in\mathbb{N}$. 
\begin{figure}[H]
\begin{tikzcd}
f \arrow[r, "\widehat{}"] \arrow[d, "C_{m,n}"] &  \left(\widehat{f}(\mathbf{k})\right)_{\mathbf{k}\in\mathbb{Z}^2} \arrow[d] \arrow[r, "\mathcal{O}",equal] & \left(\widehat{f}(\mathbf{k};2K+1)\right)_{\|\mathbf{k}\|_\infty\le K}\arrow[d]\\
C_{m,n}(f) \arrow[r,  "(\ref{C.ZBVC})",equal]& \sum_{\mathbf{k}\in\mathbb{Z}^2}\mathbf{c}(\mathbf{k};m,n)\widehat{f}(\mathbf{k})  
 \arrow[r,  "\mathcal{O}",equal] & C_{m,n}^K(f):=\sum_{\|\mathbf{k}\|_\infty\le K}\mathbf{c}(\mathbf{k};m,n)\widehat{f}(\mathbf{k};2K+1)
\end{tikzcd}
\caption{The diagram comparing construction of Fourier-Bessel transforms vs.  FFBT}
\end{figure}
\subsection{Inverse Finite Fourier-Bessel transform (iFFBT)}\label{iFFBT}
In this part,  using the notion of FFBT,  developed in Section \ref{FFBT},  we introduce the notion of inverse finite Fourier-Bessel transforms (iFFBT) as numerical approximations of Fourier-Bessel partial sums given by (\ref{SMN}),  for functions supported on disks.  It is shown that inverse finite Fourier-Bessel transforms are converging to Fourier-Bessel partial sums in the uniform sense.  The matrix form of iFFBT discussed in Section \ref{MatiFFBT}.
 
Suppose $f\in\mathcal{V}$ and $M,N\in\mathbb{N}$.  
The inverse finite Fourier-Bessel transform ($\mathbf{iFFBT}$) of order $K\in\mathbb{N}$ of $S_{M,N}(f)$ is given by 
\begin{equation}\label{iDFBT-FFT}
S_{M,N}^{K}(f)(x,y)
:=\sum_{m=-M}^M\sum_{n=1}^NC_{m,n}^{K}(f)\Psi_{m,n}(x,y),
\end{equation}
for $(x,y)\in\mathbb{R}^2$, where $C_{m,n}^K(f)$ is given by (\ref{DFBT-FFT}).

Then  
\begin{equation}\label{iDFBT-FFTalt}
S_{M,N}^{K}(f)(x,y)
=\sum_{\|\mathbf{k}\|_\infty\le K}\left(\sum_{m=-M}^M\sum_{n=1}^N\mathbf{c}(\mathbf{k};m,n)\Psi_{m,n}(x,y)\right)\widehat{f}(\mathbf{k};2K+1).
\end{equation}
Throughout, we discuss a criteria which guarantees convergence of the iFFBT to the partial sums of Fourier-Bessel expansions. 
For $M,N\in\mathbb{N}$, let 
\[
K{[M,N]}:=\max\{K_{m,n}:0\le m\le M,1\le n\le N\}.
\]

Next theorem shows that $S_{M,N}^{K}(f)$ can be considered as an accurate  uniform approximation of the partial Fourier-Bessel sum $S_{M,N}(f)$ if $K$ is proportional to $M,N$.  
\begin{theorem}\label{SMNKer}
Let $f\in\mathcal{C}^1(\mathbb{R}^2)\cap\mathcal{A}$ be a function supported in $\mathbb{B}^\circ$ and $M,N\in\mathbb{N}$.  Then there exist constants $c_f,D[M,N]>0$ such that 
\[
\left\|S_{M,N}^{K}(f)-S_{M,N}(f)\right\|_\infty\le\frac{2c_fD[M,N]}{K},\hspace{1.5cm}{\rm for}\ K\ge K[M,N].
\] 
\end{theorem}
\begin{proof}
Suppose that $K\ge K[M,N]$.  Then $K\ge K_{m,n}$ and hence using Theorem \ref{CmnKerr}, we achieve 
\begin{equation}\label{SMNKerr.alt}
\left|C_{m,n}^{K}(f)-C_{m,n}(f)\right|\le\frac{2\gamma_{m,n}c_f}{K},
\end{equation}
for every $0\le|m|\le M$ and $1\le n\le N$. Therefore,  using (\ref{SMNKerr.alt}), we obtain 
\begin{align*}
&\left|S_{M,N}^{K}(f)(x,y)-S_{M,N}(f)(x,y)\right|
\le\sum_{m=-M}^M\sum_{n=1}^N\left|C_{m,n}^{K}(f)-C_{m,n}(f)\right|\left|\Psi_{m,n}(x,y)\right|
\\&\le\frac{2c_f}{K}\sum_{m=-M}^M\sum_{n=1}^N\gamma_{m,n}\left|\Psi_{m,n}(x,y)\right|
\le\frac{2c_f}{K}\sum_{m=-M}^M\sum_{n=1}^N\frac{\gamma_{m,n}}{|J_{m+1}(z_{m,n})|}=:\frac{2c_fD[M,N]}{K},
\end{align*}
for every $(x,y)\in\Omega$. So, we get 
\[
\left\|S_{M,N}^{K}(f)-S_{M,N}(f)\right\|_\infty=\sup_{(x,y)\in\Omega}\left|S_{M,N}^{K}(f)(x,y)-S_{M,N}(f)(x,y)\right|\le\frac{2c_f D[M,N]}{K}.
\]
\end{proof}
\begin{remark}
Let $f\in \mathcal{C}^1(\mathbb{R}^2)\cap\mathcal{A}$ be supported in $\mathbb{B}^\circ$ and $M,N\in\mathbb{N}$.  Suppose $\epsilon>0$ is given.  Assume $K:=\lceil \max\{\epsilon^{-1}2c_fD[M,N],K[M,N]\}\rceil$. Then $S_{M,N}^K(f)$ uniformly approximates the Fourier-Bessel partial sum $S_{M,N}(f)$ with the absolute error less than $\epsilon$. 
\end{remark}
\begin{corollary}
{\it Let $f\in\mathcal{C}^2(\mathbb{R}^2)$ be a function supported in $\mathbb{B}^\circ$ and $M,N\in\mathbb{N}$.  Then 
\[
\left\|S_{M,N}^{K}(f)-S_{M,N}(f)\right\|_\infty=\mathcal{O}(1/K),\hspace{1.5cm}{\rm for}\ K\ge K[M,N].
\] 
}\end{corollary}
\begin{remark}\label{iFFBTstp}
The iFFBT satisfies the following asymptotic steerability.  Let $f\in\mathcal{C}^2(\mathbb{R}^2)$ be supported in $\mathbb{B}^\circ$.  Suppose $M,N\in\mathbb{N}$ and $\phi\in[0,2\pi)$.  Then 
\[
\left\|S_{M,N}^{K}(R_\phi f)-R_\phi S_{M,N}^{K}(f)\right\|_\infty=\mathcal{O}(1/K),  \hspace{1.5cm}{\rm for}\ K\ge K[M,N].
\]
\end{remark}
\begin{remark}\label{Ba}
The iFFBT numerical scheme can be canonically applied for approximation of functions supported in arbitrary disks by scaling. Let $a>0$ and $f:\mathbb{R}^2\to\mathbb{C}$ be an integrable function supported in $\mathbb{B}_a$.  Then $\widetilde{f}:\mathbb{R}^2\to\mathbb{C}$ given by $\widetilde{f}(\mathbf{x}):=f(a\mathbf{x})$ is supported in $\mathbb{B}$.  Assume that $S_{M,N}(\widetilde{f})$ is an approximation of $\widetilde{f}$ in $L^2$-sense for large enough $M,N$. Then $$S_{M,N}^{a,K}(f)(\mathbf{x}):=S_{M,N}^K(\widetilde{f})(a^{-1}\mathbf{x})$$ is a numerical approximation of $f(\mathbf{x})$ in $L^2$-sense, if $K\ge K[M,N]$.
\end{remark}
\section{\bf Numerics of FFBT/iFFBT in MATLAB}

This section discusses a unified strategy in terms of matrix functionals including matrix multiplication and matrix traces for numerical implementation of finite Fourier-Bessel transforms using fast Fourier algorithms in MATLAB.

\subsection{Matrix form of finite Fourier transforms on disks}\label{MatFFT}
This part presents the matrix form of numerical Fourier integral on disks,   discussed in  Section \ref{FFtD}.  Let $f:\mathbb{R}^2\to\mathbb{C}$ be supported in $\mathbb{B}$,  $\mathbf{k}:=(k_1,k_2)^T\in\mathbb{Z}^2$, and $L\in\mathbb{N}$.  Suppose $x_i:=-1+(i-1)\delta$ with $\delta:=2/L$ and $1\le i\le L$.
So,  
\begin{align*}
\widehat{f}(\mathbf{k};L)&=\delta^2\sum_{i=1}^L\sum_{j=1}^Lf(x_i,x_j)e^{-\pi\ii (k_1x_i+k_2x_j)}
\\&=\delta^2(-1)^{k_1+k_2}\sum_{i=1}^L\sum_{j=1}^Lf(x_i,x_j)e^{-2\pi\ii k_1(i-1)/L}e^{-2\pi\ii k_2(j-1)/L}
\\&=\delta^2(-1)^{k_1+k_2}\sum_{i=1}^L\sum_{j=1}^Lf(x_i,x_j)e^{-2\pi\ii\tau_L(k_1)(i-1)/L}e^{-2\pi\ii\tau_L(k_2)(j-1)/L}
\\&=\delta^2(-1)^{k_1+k_2}\widehat{\mathbf{F}}(\tau_L(k_1)+1,\tau_L(k_2)+1),
\end{align*}
where the matrix $\mathbf{F}\in\mathbb{C}^{L\times L}$ is given by 
$\mathbf{F}(i,j):=f(x_i,x_j)$ for every $1\le i,j\le L$, $\widehat{\mathbf{F}}$ is the DFT of $\mathbf{F}$ given by (\ref{DFT}) and $\tau_L(k_\imath):=\mathrm{mod}(k_\imath,L)$ for $\imath=1,2$.

We here introduce Algorithm \ref{Xi.kK} using regular sampling on $\Omega$ to approximate the Fourier integrals of the form $\widehat{f}(\mathbf{k})$ for function $f$ supported on $\mathbb{B}$ with respect to a given absolute error.
\begin{algorithm}[H]
\caption{Finding $\epsilon$-approximation of $\widehat{f}(\mathbf{k})$ using regular square sampling of $\Omega$ and FFT} 
\begin{algorithmic}[1]
\State{\bf input data} The function $f:\mathbb{R}^2\to\mathbb{C}$ supported in $\mathbb{B}$,  given error $\epsilon>0$ and $\mathbf{k}:=(k_1,k_2)^T\in\mathbb{Z}^2$
\State {\bf output result} $\widehat{f}(\mathbf{k},L)$ with the absolute error $\le\epsilon$\\ 
Put $\beta_\epsilon:=\max\{\epsilon^{-1}\pi^{-1}96\|\nabla f\|_\infty,\|\mathbf{k}\|_\infty\}$,  find $K\in\mathbb{N}$ such that $K\ge\beta_\epsilon$ and let  $L:=2K+1$\\
Generate the uniform square sampling gird $(x_i,x_j)$ with $x_i:=-1+(i-1)\delta$,  $\delta:=\frac{2}{L}$, and $1\le i\le L$\\
Generate sampled values $\mathbf{F}:=(f(x_i,x_j))_{1\le i,j\le L}$ 	of the 
	function $f:\mathbb{R}^2\to\mathbb{C}$.
\State Compute 
$$\widehat{f}(\mathbf{k};L):=\delta^2(-1)^{k_1+k_2}\widehat{\mathbf{F}}(\tau_L(k_1)+1,\tau_L(k_2)+1).$$
\end{algorithmic} 
\label{Xi.kK}
\end{algorithm}
\subsection{Matrix form of FFBT}\label{MatFFBT}
Throughout,  we discuss the matrix form of the finite Fourier-Bessel transform given in (\ref{DFBT-FFT}).   Assume $(m,n)\in\mathbb{Z}\times\mathbb{N}$.  For simplicity,  we may use $\mathbf{c}(k_1,k_2)$ instead of $\mathbf{c}(k_1,k_2;m,n)$ for $k_1,k_2\in\mathbb{Z}$.
Let $K\in\mathbb{N}$ and $L:=2K+1$.  Assume $f\in\mathcal{V}$ and $\mathbf{F}\in\mathbb{C}^{L\times L}$ is given by $\mathbf{F}(i,j):=f(x_i,x_j)$ where $x_j:=-1+(j-1)\delta$ with $\delta:=\frac{2}{L}$. 
Suppose $-K\le k_2\le K$ is arbitrary. 
Then
\begin{align*}
&\sum_{k_1=-K}^{K}\mathbf{c}(k_1,k_2;m,n)(-1)^{k_1}\widehat{\mathbf{F}}(\tau_L(k_1)+1,\tau_L(k_2)+1)
\\&=\sum_{k_1=-K}^{-1}\mathbf{c}(k_1,k_2)(-1)^{k_1}\widehat{\mathbf{F}}(L+k_1+1,\tau_L(k_2)+1)
+\sum_{k_1=0}^{K}\mathbf{c}(k_1,k_2)(-1)^{k_1}\widehat{\mathbf{F}}(k_1+1,\tau_L(k_2)+1)
\\&=\sum_{k_1=K+2}^{L}\mathbf{c}(k_1-L-1,k_2)(-1)^{k_1}\widehat{\mathbf{F}}(k_1,\tau_L(k_2)+1)
+\sum_{k_1=1}^{K+1}\mathbf{c}(k_1-1,k_2)(-1)^{k_1-1}\widehat{\mathbf{F}}(k_1,\tau_L(k_2)+1)
\\&=\sum_{l=1}^LQ_{k_2}(m,n;l)\widehat{\mathbf{F}}(l,\tau_L(k_2)+1),
\end{align*}
where 
\begin{equation}\label{Qk}
Q_{k_2}(m,n;l):=\left\{\begin{array}{lll}
\mathbf{c}(l-1,k_2;m,n)(-1)^{l-1} & {\rm if}\ 1\le l\le K+1\\
\mathbf{c}(l-L-1,k_2;m,n)(-1)^{l} & {\rm if}\ K+2\le l\le L\\
\end{array}\right..
\end{equation}
Therefore,  we get 
\begin{align*}
&C_{m,n}^{K}(f)=\sum_{\|\mathbf{k}\|_{\infty}\le K}\mathbf{c}(\mathbf{k};m,n)\widehat{f}(\mathbf{k};L)
=\sum_{k_1=-K}^{K}\sum_{k_2=-K}^K\mathbf{c}(k_1,k_2;m,n)\widehat{f}( k_1, k_2;L)
\\&=\delta^2\sum_{k_1=-K}^{K}\sum_{k_2=-K}^K\mathbf{c}(k_1,k_2;m,n)(-1)^{k_1+k_2}\widehat{\mathbf{F}}(\tau_L(k_1)+1,\tau_L(k_2)+1)
\\&=\delta^2\sum_{k_2=-K}^K(-1)^{k_2}\sum_{k_1=-K}^{K}\mathbf{c}(k_1,k_2;m,n)(-1)^{k_1}\widehat{\mathbf{F}}(\tau_L(k_1)+1,\tau_L(k_2)+1)
\\&=\delta^2\sum_{k_2=-K}^K(-1)^{k_2}\sum_{l=1}^LQ_{k_2}(m,n;l)\widehat{\mathbf{F}}(l,\tau_L(k_2)+1)
=\delta^2\sum_{l=1}^L\sum_{k_2=-K}^K(-1)^{k_2}Q_{k_2}(m,n;l)\widehat{\mathbf{F}}(l,\tau_L(k_2)+1).
\end{align*}
Assume $1\le l\le L$.  Then,  using similar argument, we get 
\begin{align*}
&\sum_{k_2=-K}^K(-1)^{k_2}Q_{k_2}(m,n;l)\widehat{\mathbf{F}}(l,\tau_L(k_2)+1)
\\&=\sum_{k_2=-K}^{-1}(-1)^{k_2}Q_{k_2}(m,n;l)\widehat{\mathbf{F}}(l,L+k_2+1)+\sum_{k_2=0}^K(-1)^{k_2}Q_{k_2}(m,n;l)\widehat{\mathbf{F}}(l,k_2+1)
\\&=\sum_{k_2=K+2}^{L}(-1)^{k_2}Q_{k_2-L-1}(m,n;l)\widehat{\mathbf{F}}(l,k_2)+\sum_{k_2=1}^{K+1}(-1)^{k_2-1}Q_{k_2-1}(m,n;l)\widehat{\mathbf{F}}(l,k_2)
=\sum_{\ell=1}^LQ(m,n;l,\ell)\widehat{\mathbf{F}}(l,\ell),
\end{align*}
where 
\[
Q(m,n;l,\ell):=\left\{\begin{array}{lll}
Q_{\ell-1}(m,n;l)(-1)^{\ell-1} & {\rm if}\ 1\le \ell\le K+1\\
Q_{\ell-L-1}(m,n;l)(-1)^{\ell} & {\rm if}\ K+2\le \ell\le L\\
\end{array}\right..
\]
So, we obtain 
\begin{equation}\label{CmnaK.LL}
C_{m,n}^{K}(f)=\delta^2\sum_{l=1}^L\sum_{k_2=-K}^K(-1)^{k_2}Q_{k_2}(m,n;l)\widehat{\mathbf{F}}(l,\tau_L(k_2)+1)
=\delta^2\sum_{l=1}^L\sum_{\ell=1}^LQ(m,n;l,\ell)\widehat{\mathbf{F}}(l,\ell).
\end{equation}
In addition,  by applying (\ref{Qk}) we get 
\begin{equation}
Q(m,n;l,\ell)=\left\{\begin{array}{llll}
\mathbf{c}(l-1,\ell-1;m,n)(-1)^{l+\ell} & {\rm if}\ 1\le \ell\le K+1,\ 1\le l\le K+1\\
\mathbf{c}(l-L-1,\ell-1;m,n)(-1)^{l+\ell-1}& {\rm if}\ 1\le \ell\le K+1,\ K+2\le l\le L\\
\mathbf{c}(l-1,\ell-L-1;m,n)(-1)^{l+\ell-1}  & {\rm if}\ K+2\le \ell\le L,  \ 1\le l\le K+1\\
\mathbf{c}(l-L-1,\ell-L-1;m,n)(-1)^{l+\ell} & {\rm if}\ K+2\le \ell\le L,\ K+2\le l\le L\\
\end{array}\right..
\end{equation} 
Assume $\mathbf{Q}(m,n,:,:)\in\mathbb{C}^{L\times L}$ is given by 
$$\mathbf{Q}(m,n,\ell,l):=Q(m,n;l,\ell), $$ for every $1\le \ell,l\le L$.  Then
\begin{equation}\label{MAT.DFBT-FFT}
C_{m,n}^{K}(f)=\delta^2\mathrm{tr}[\mathbf{Q}(m,n,:,:)\widehat{\mathbf{F}}].
\end{equation}
Indeed,  using (\ref{CmnaK.LL}),  we can write 
\begin{align*}
C_{m,n}^{K}(f)
&=\delta^2\sum_{\ell=1}^L\sum_{l=1}^L\mathbf{Q}(m,n,\ell,l)\mathbf{\widehat{F}}(l,\ell)
=\delta^2\sum_{\ell=1}^L(\mathbf{Q}(m,n,:,:)\mathbf{\widehat{F}})_{\ell,\ell}=\delta^2\mathrm{tr}[\mathbf{Q}(m,n,:,:)\widehat{\mathbf{F}}].
\end{align*}

We here give Algorithm \ref{DFBT-FFTalgE} using regular sampling to approximate the Fourier-Bessel coefficient $C_{m,n}(f)$ of a function $f:\mathbb{R}^2\to\mathbb{C}$ supported in $\mathbb{B}$ with respect to a given absolute error.
\begin{algorithm}[H]
\caption{Finding $\epsilon$-approximation of $C_{m,n}(f)$ using $C_{m,n}^{K}(f)$} 
\begin{algorithmic}[1]
\State{\bf input data} The function $f:\mathbb{R}^2\to\mathbb{C}$ supported in $\mathbb{B}^\circ$,  given error $\epsilon>0$ and $(m,n)\in\mathbb{Z}\times\mathbb{N}$
\State {\bf output result} $C_{m,n}^{K}(f)$ with the absolute error $\le\epsilon$\\ 
Let $K_{m,n}:=\lceil\frac{z_{m,n}}{\pi}\rceil$ and put $\beta_\epsilon:=\max\{2\epsilon^{-1}c_f\gamma_{m,n},K_{m,n}\}$.\\
	Find $K\in\mathbb{N}$ such that $K\ge\beta_\epsilon$ and let $L:=2K+1$.\\
Load precomputed data $\mathbf{Q}(m,n,:,:)\in\mathbb{C}^{L\times L}$\\
Generate the sampling grid $(x_i,x_j)$ with $x_i:=-1+(i-1)\delta$,  $\delta:=\frac{2}{L}$, and $1\le i\le L$\\
Generate sampled values $\mathbf{F}:=(f(x_i,x_j))_{1\le i,j\le L}$ on the uniform grid of $\Omega$.
\State Compute the DFT $\widehat{\mathbf{F}}$.
\State Compute 
$$C_{m,n}^{K}(f)=\delta^2\mathrm{tr}[\mathbf{Q}(m,n,:,:)\widehat{\mathbf{F}}].$$
\end{algorithmic} 
\label{DFBT-FFTalgE}
\end{algorithm}
\subsection{Matrix form of iFFBT}\label{MatiFFBT}
Then we discuss the matrix form for iFFBT using finite Fourier transform and DFT in terms of matrix multiplication and matrix trace.  

Let $f\in\mathcal{V}$ and $M,N\in\mathbb{N}$.  Suppose $K\in\mathbb{N}$ and $L:=2K+1$.  Let $x_i:=-1+(i-1)\delta$ with $\delta:=\frac{2}{L}$ for $1\le i\le L$.  Let $\mathbf{F}\in\mathbb{C}^{L\times L}$ be given by $\mathbf{F}(i,j):=f(x_i,x_j)$.  Assume $\widehat{\mathbf{F}}$ is the DFT of $\mathbf{F}$ and $(x,y)\in\Omega$.  Then (\ref{iDFBT-FFTalt}) and (\ref{CmnaK.LL}) imply that  
\begin{equation}\label{iFFBTaltMat}
S_{M,N}^{K}(f)(x,y)
=\delta^2\sum_{l=1}^{L}\sum_{\ell=1}^{L}\left(\sum_{m=-M}^M\sum_{n=1}^NQ(m,n;l,\ell)\Psi_{m,n}(x,y)\right)\widehat{\mathbf{F}}(l,\ell).
\end{equation}
Hence,  if $1\le l,\ell\le L$,  we get  
\begin{align*}
&\sum_{m=-M}^M\sum_{n=1}^NQ(m,n;l,\ell)\Psi_{m,n}(x,y)=\sum_{t=1}^{2M+1}\sum_{n=1}^NQ(t-M-1,n;l,\ell)\Psi_{t-M-1,n}(x,y)
\\&=\sum_{t=1}^{2M+1}\sum_{n=1}^N\mathbf{H}(l,\ell,:,:)_{t,n}(\mathbf{P}(x,y)^T)_{n,t}
=\sum_{t=1}^{2M+1}(\mathbf{H}(l,\ell,:,:)\mathbf{P}(x,y)^T)_{t,t}=\mathrm{tr}\left[\mathbf{H}(l,\ell,:,:)\mathbf{P}(x,y)^T\right],
\end{align*}
where the array $\mathbf{H}\in\mathbb{C}^{L\times L\times(2M+1)\times N}$ and 
the matrix $\mathbf{P}(x,y)\in\mathbb{C}^{(2M+1)\times N}$ are given by 
\begin{equation*}\label{Hmat+PaxyMain}
\mathbf{H}(l,\ell,t,n):=Q(t-M-1,n,l,\ell), \hspace{1cm}\mathbf{P}(x,y)_{t,n}:=
		\Psi_{t-1-M,n}(x,y).
\end{equation*}
Therefore,  using (\ref{iFFBTaltMat}),  we obtain 
\begin{align*}
S_{M,N}^K(f)(x,y)
&=\delta^2\sum_{l=1}^{L}\sum_{\ell=1}^{L}\mathrm{tr}\left[\mathbf{H}(l,\ell,:,:)\mathbf{P}(x,y)^T\right]\widehat{\mathbf{F}}(l,\ell)
=\delta^2\sum_{l=1}^{L}\sum_{\ell=1}^{L}\mathbf{K}(x,y)_{\ell,l}\widehat{\mathbf{F}}(l,\ell)
\\&=\delta^2\sum_{\ell=1}^{L}\sum_{l=1}^{L}\mathbf{K}(x,y)_{\ell,l}\widehat{\mathbf{F}}(l,\ell)
=\delta^2\sum_{\ell=1}^{L}\left(\mathbf{K}(x,y)\widehat{\mathbf{F}}\right)_{\ell,\ell}=\delta^2\mathrm{tr}\left[\mathbf{K}(x,y)\widehat{\mathbf{F}}\right],
\end{align*}
implying the following matrix form  
\begin{equation}\label{SMNKxy.tr}
S_{M,N}^{K}(f)(x,y)=\delta^2\mathrm{tr}\left[\mathbf{K}(x,y)\widehat{\mathbf{F}}\right],
\end{equation}
where the square matrix $\mathbf{K}(x,y)\in\mathbb{C}^{L\times L}$ is defined by 
\[
\mathbf{K}(x,y)_{\ell,l}:=\mathrm{tr}\left[\mathbf{H}(l,\ell,:,:)\mathbf{P}(x,y)^T\right],\hspace{1cm}\ {\rm for}\ 1\le l,\ell\le L.
\]

We here give Algorithm \ref{iDFBT-FFTalg} using regular sampling to approximate the Fourier-Bessel partial sum $S_{M,N}(f)$ of a function $f$ supported in $\mathbb{B}$ with respect to a given absolute error by $S_{M,N}^{K}(f)$.
\begin{algorithm}[H]
\caption{Finding $\epsilon$-approximation of $S_{M,N}(f)(x,y)$ using $S_{M,N}^K(f)(x,y)$} 
\begin{algorithmic}[1]
\State{\bf input data} The function $f$ supported in $\mathbb{B}^\circ$,  $(x,y)\in\mathbb{B}$,  given error $\epsilon>0$,  and $M,N\in\mathbb{N}$
\State {\bf output result} $S_{M,N}^{K}(f)(x,y)$ with the absolute error $\le\epsilon$\\ 
Compute 
$$D[M,N]:=\frac{1}{\sqrt{2}}\sum_{m=-M}^M\sum_{n=1}^N\frac{\gamma_{m,n}}{|J_{m+1}(z_{m,n})|}.
$$
\State Compute $
K[M,N]:=\max\{K_{m,n}:0\le m\le M,1\le n\le N\}.$
\State Compute $\beta_\epsilon:=\max\{2\epsilon^{-1}c_fD[M,N],K[M,N]\}$,
\State Find $K\in\mathbb{N}$ such that $K\ge\beta_\epsilon$, and let $L:=2K+1$.
\State Load the precomputed values $\mathbf{P}(x,y)\in\mathbb{C}^{(2M+1)\times N}$ and $\mathbf{H}\in\mathbb{C}^{L\times L\times(2M+1)\times N}$
\State Generate the square matrix $\mathbf{K}(x,y)\in\mathbb{C}^{L\times L}$ using 
\[
\mathbf{K}(x,y)_{\ell,l}:=\mathrm{tr}\left[\mathbf{H}(l,\ell,:,:)\mathbf{P}(x,y)^T\right],\hspace{1cm}\ {\rm for}\ 1\le l,\ell\le L.
\]
\State Generate the sampling grid $(x_i,x_j)$ with $x_i:=-1+(i-1)\delta$,  $\delta:=\frac{2}{L}$,  and $1\le i,j\le L$\\
Generate sampled values $\mathbf{F}:=(f(x_i,x_j))_{1\le i,j\le L}$ 	of the function $f:\mathbb{R}^2\to\mathbb{C}$.
\State Compute the DFT matrix $\widehat{\mathbf{F}}$
\State Compute 
$$S_{M,N}^{K}(f)(x,y):=\delta^2\mathrm{tr}\left[\mathbf{K}(x,y)\widehat{\mathbf{F}}\right].$$
\end{algorithmic} 
\label{iDFBT-FFTalg}
\end{algorithm}

\subsection{Numerical experiments} We here implement some experiments in MATLAB for functions supported on disks using the discussed matrix approach (\ref{SMNKxy.tr}).  

The matrix closed forms (\ref{MAT.DFBT-FFT}) and (\ref{SMNKxy.tr}) have practical advantages in numerical experiments and computations.  To begin with,  the matrix forms reduce multivariate summations in the right hand-side of (\ref{DFBT-FFT}) and (\ref{iDFBT-FFT}) into matrix trace and matrix multiplication which can be computed using fast matrix multiplications algorithms.  In addition,  the matrix forms (\ref{MAT.DFBT-FFT}) and (\ref{SMNKxy.tr}) involve the discrete Fourier transform (DFT) of sampled values matrix of the function which can be then performed using fast Fourier algorithms. 

To start with, we consider functions which are smoothly supported (absolutely/approximately) in disks.  Then we consider functions supported in disks constructed by zero-padding of another functions.  Next we consider characteristics/indicator functions of 2D bodies.

\newpage 
A large class of smooth functions which are absolutely zero on the boundary and outside of disks are finite linear combinations of normalized polar harmonics given by (\ref{2D.B.Fr.gp}).

\begin{example} Suppose $f:=\Psi_{1,2}+\Psi_{2,1}$. Then $f\in\mathcal{C}^\infty(\mathbb{R}^2)$ and $f$ is supported in $\mathbb{B}^\circ$. 
\begin{figure}[H]
\centering
\includegraphics[keepaspectratio=true,width=\textwidth, height=0.45\textheight]{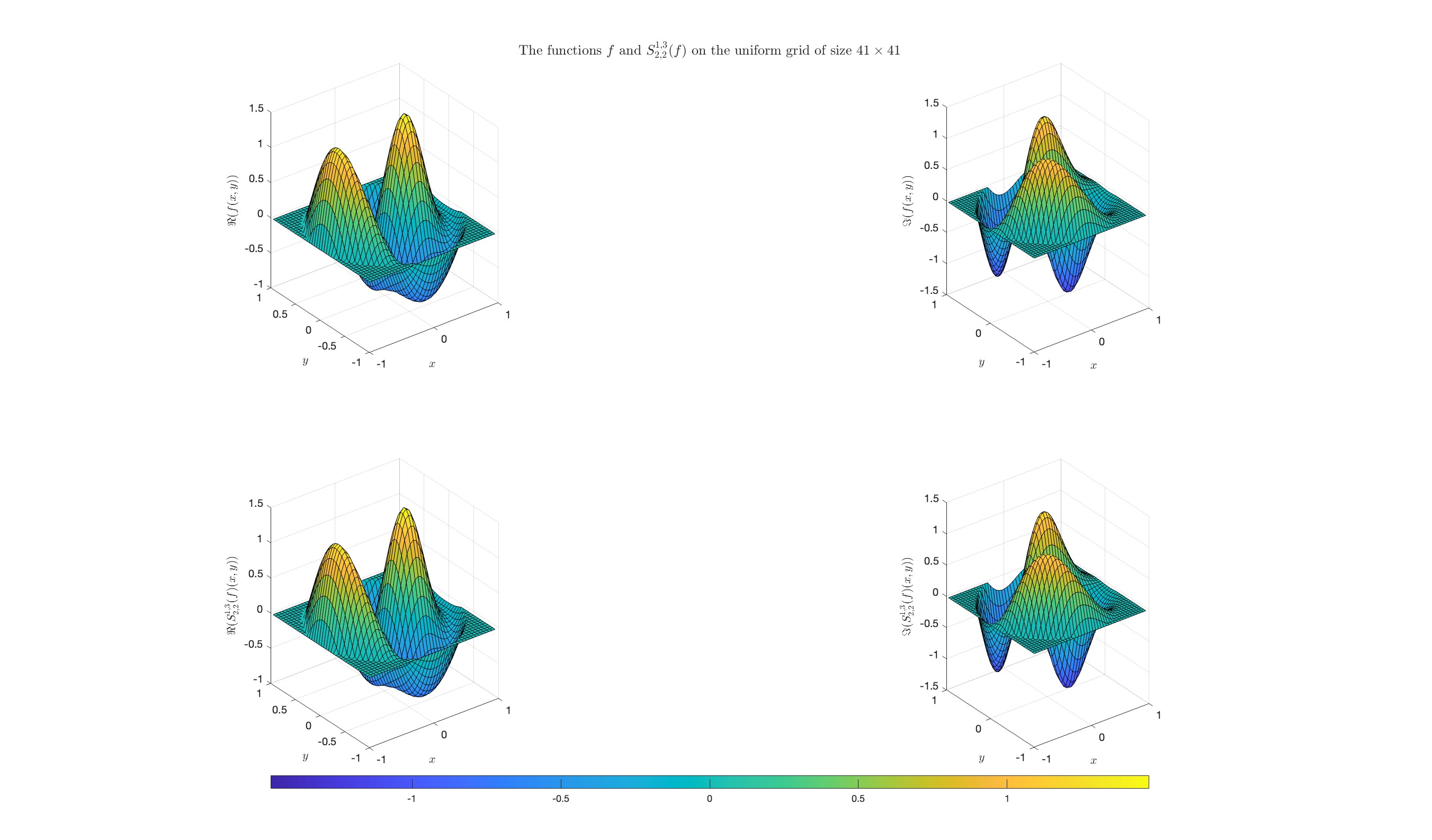}
\caption{The 3D surface plot of $f$ and $S_{2,2}^{3}(f)$ where $f:=\Psi_{1,2}+\Psi_{2,1}$ on the uniform grid of size $41\times 41$ of $\Omega$.  It is worth noting that $K[2,2]=3$}
\label{fig:P1221Surf}
\end{figure}
\begin{figure}[H]
\centering
\includegraphics[keepaspectratio=true,width=\textwidth, height=0.35\textheight]{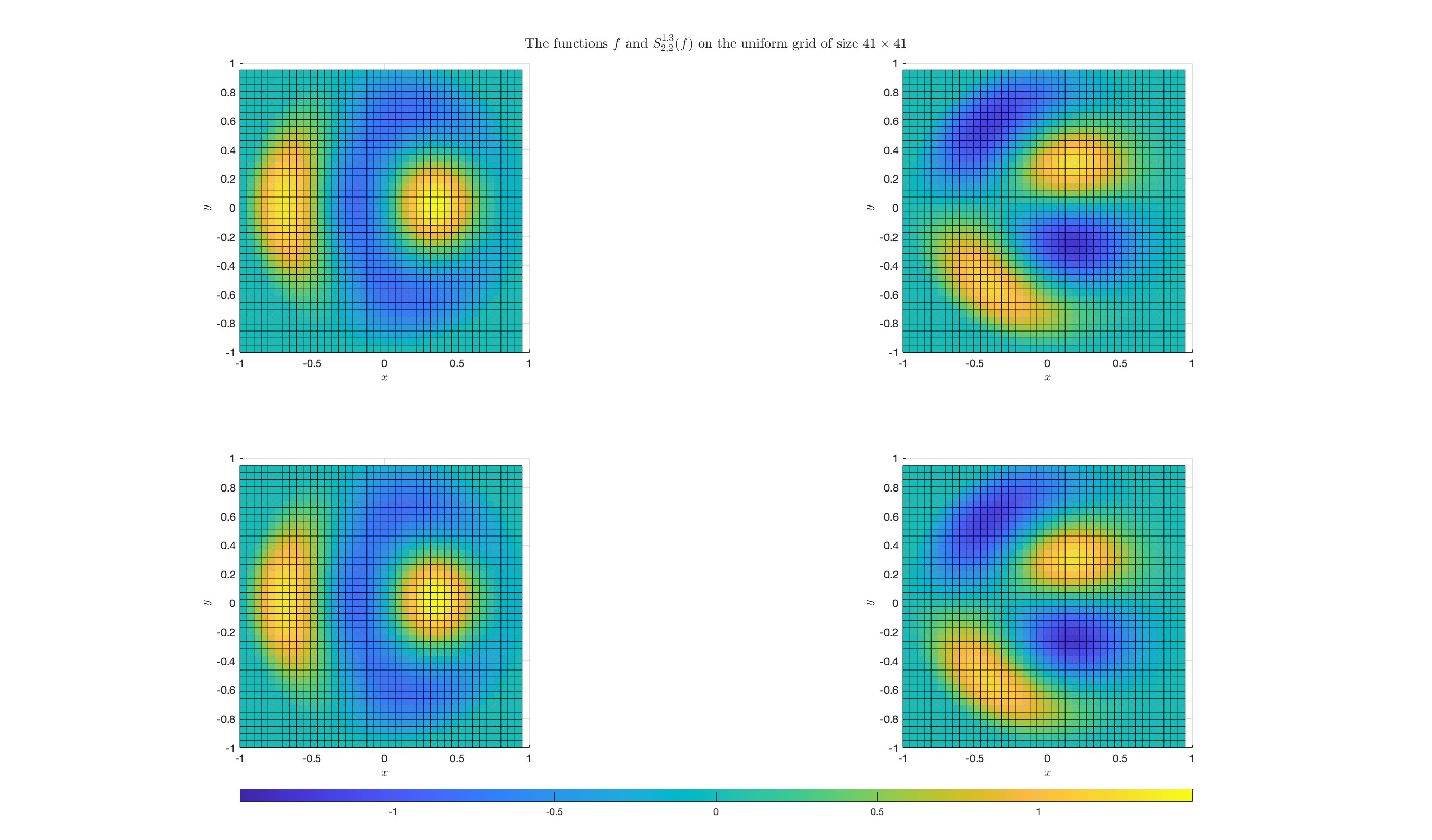}
\caption{The contour plot of $f$ and $S_{2,2}^{3}(f)$ on the uniform grid of size $41\times 41$ of $\Omega$.}
\label{fig:P1221Cont}
\end{figure}
\end{example}
Suppose that $\mathbf{A},\mathbf{B}\in M_2(\mathbb{R})$ are positive-definite.  Let $f_{\mathbf{A},\mathbf{B}}:\mathbb{R}^2\to\mathbb{C}$ be given by 
\begin{equation}\label{UAB}
f_{\mathbf{A},\mathbf{B}}(\mathbf{x}):=e^{-\mathbf{x}^T\mathbf{A}\mathbf{x}}+\ii e^{-\mathbf{x}^T\mathbf{B}\mathbf{x}}
\end{equation}
Another class of smooth functions which are approximately zero on the boundary and outside of disks are normalized 2D Gaussian of the form  (\ref{UAB}).    

\begin{example}
Let $\mathbf{A}:=\mathrm{diag}(0.1,0.05)^{-1}$ and $\mathbf{B}:=\mathrm{diag}(0.05,0.1)^{-1}$. Then $f_{\mathbf{A},\mathbf{B}}\in\mathcal{C}^\infty(\mathbb{R}^2)$ and it is approximately zero on the boundary and outside of the unit disk $\mathbb{B}$.

\begin{figure}[H]
\centering
\includegraphics[keepaspectratio=true,width=\textwidth, height=0.35\textheight]{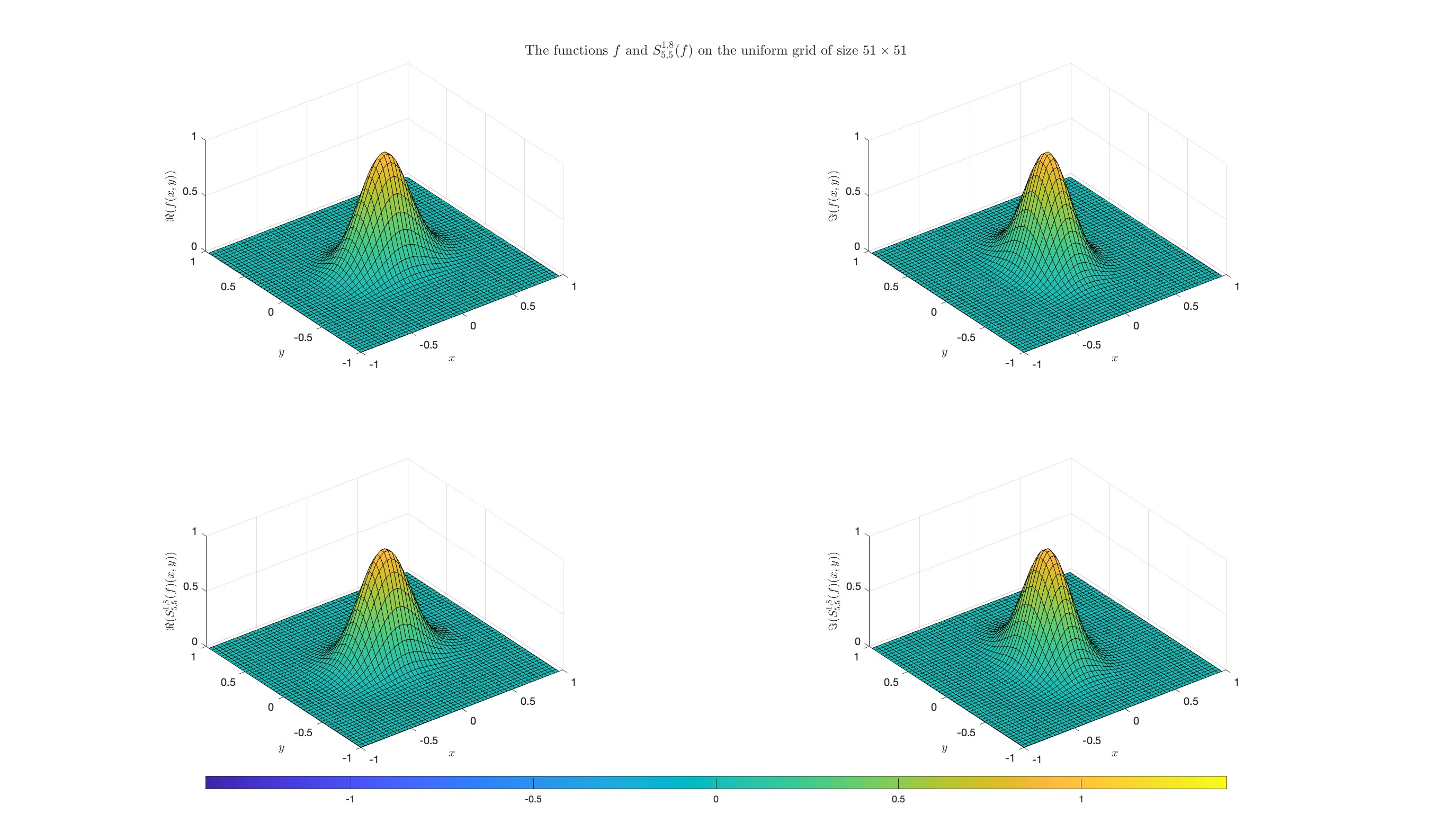}
\caption{The 3D surface plot of $f_{\mathbf{A},\mathbf{B}}$ and $S_{5,5}^{8}(f_{\mathbf{A},\mathbf{B}})$ on the uniform grid of size $51\times 51$.  It is worth noting that $K[5,5]=8$.}
\label{fig:UABSurf}
\end{figure}

\begin{figure}[H]
\centering
\includegraphics[keepaspectratio=true,width=\textwidth, height=0.35\textheight]{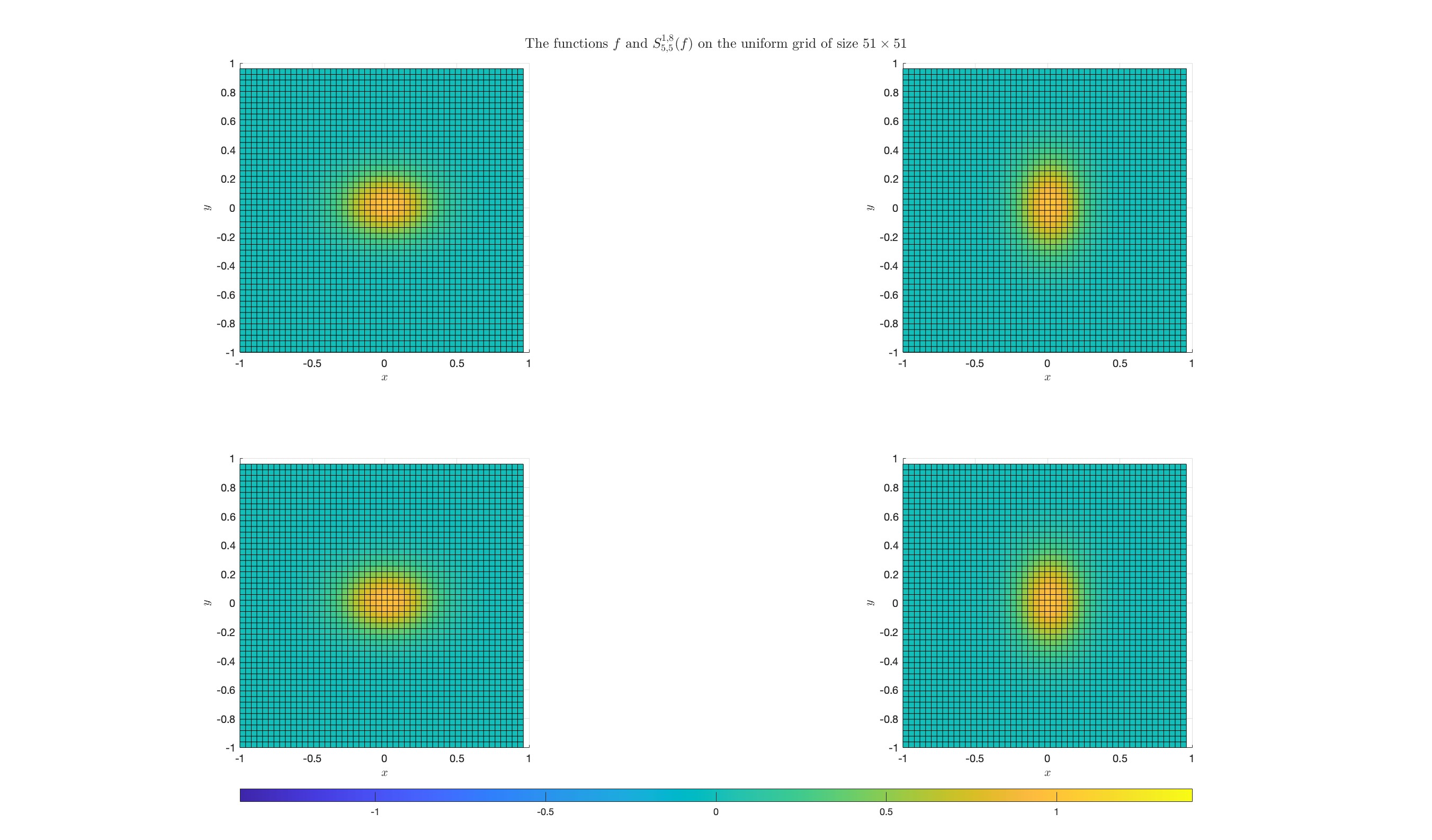}
\caption{The contour plot of $f_{\mathbf{A},\mathbf{B}}$ and $S_{5,5}^{8}(f_{\mathbf{A},\mathbf{B}})$ on the uniform grid of size $51\times 51$.}
\label{fig:UABCont}
\end{figure}
\end{example}

Next, we consider almost everywhere smooth functions which are absolutely supported in disks.  Since boundary of a disk has zero Lebesgue measure in $\mathbb{R}^2$,  such functions can be naturally obtained by zero-padding of smooth functions on $\mathbb{R}^2$ outside disks.

\begin{example}
Let $f:\mathbb{R}^2\to\mathbb{C}$ be given by 
\[
f(x,y):=\left\{\begin{array}{cc}
e^{-xy}+\ii\sin(xy) & (x,y)\in\mathbb{B} \\ 
0 & (x,y)\not\in\mathbb{B}
\end{array}\right.
\]
\begin{figure}[H]
\centering
\includegraphics[keepaspectratio=true,width=\textwidth, height=0.35\textheight]{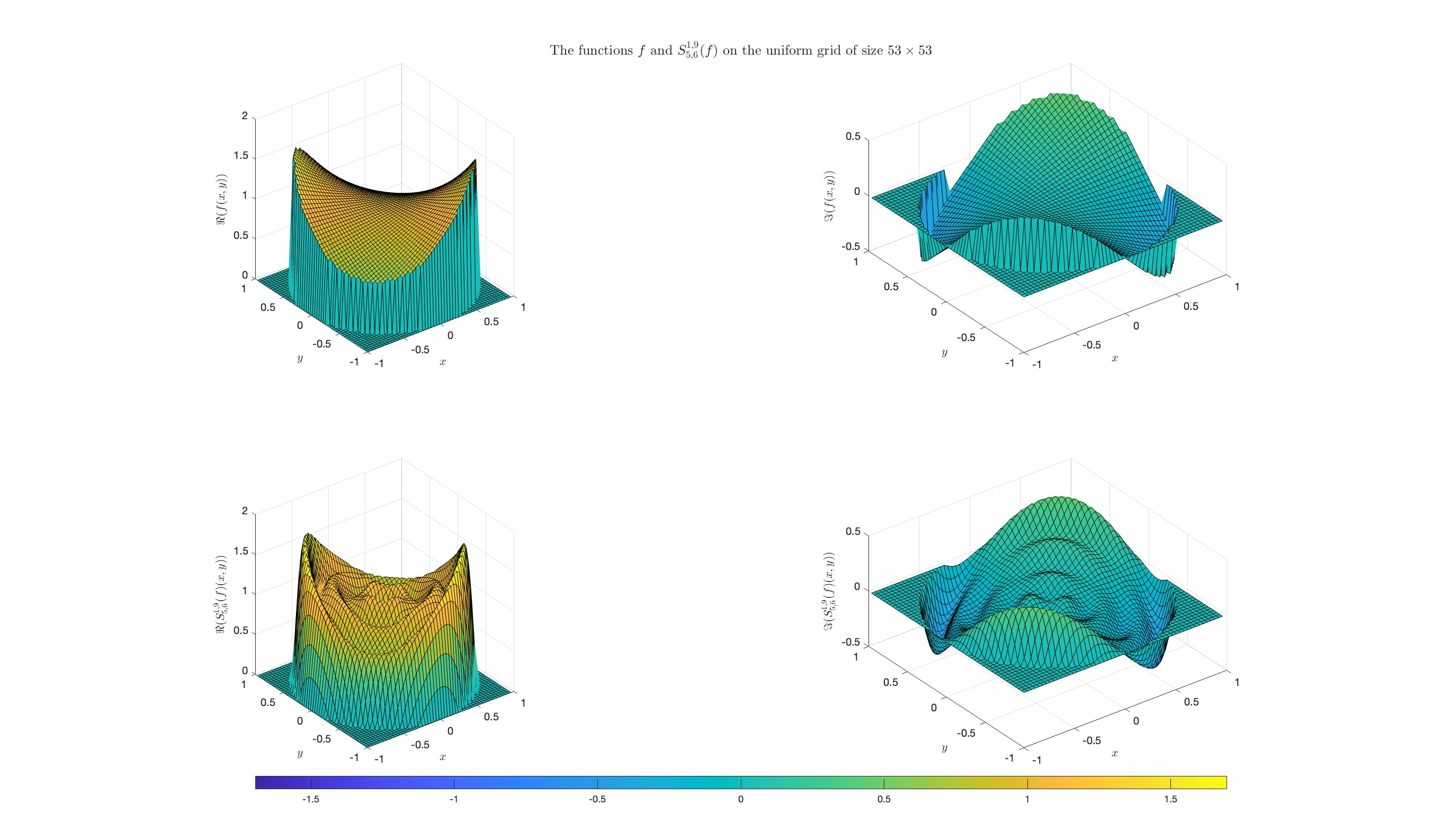}
\caption{The 3D surface plot of $f$ and $S_{5,6}^{9}(f)$ on the uniform grid of size $53\times 53$.  It is worth noting that $K[5,6]=9$.}
\label{fig:ExpSinSurf}
\end{figure}

\begin{figure}[H]
\centering
\includegraphics[keepaspectratio=true,width=\textwidth, height=0.35\textheight]{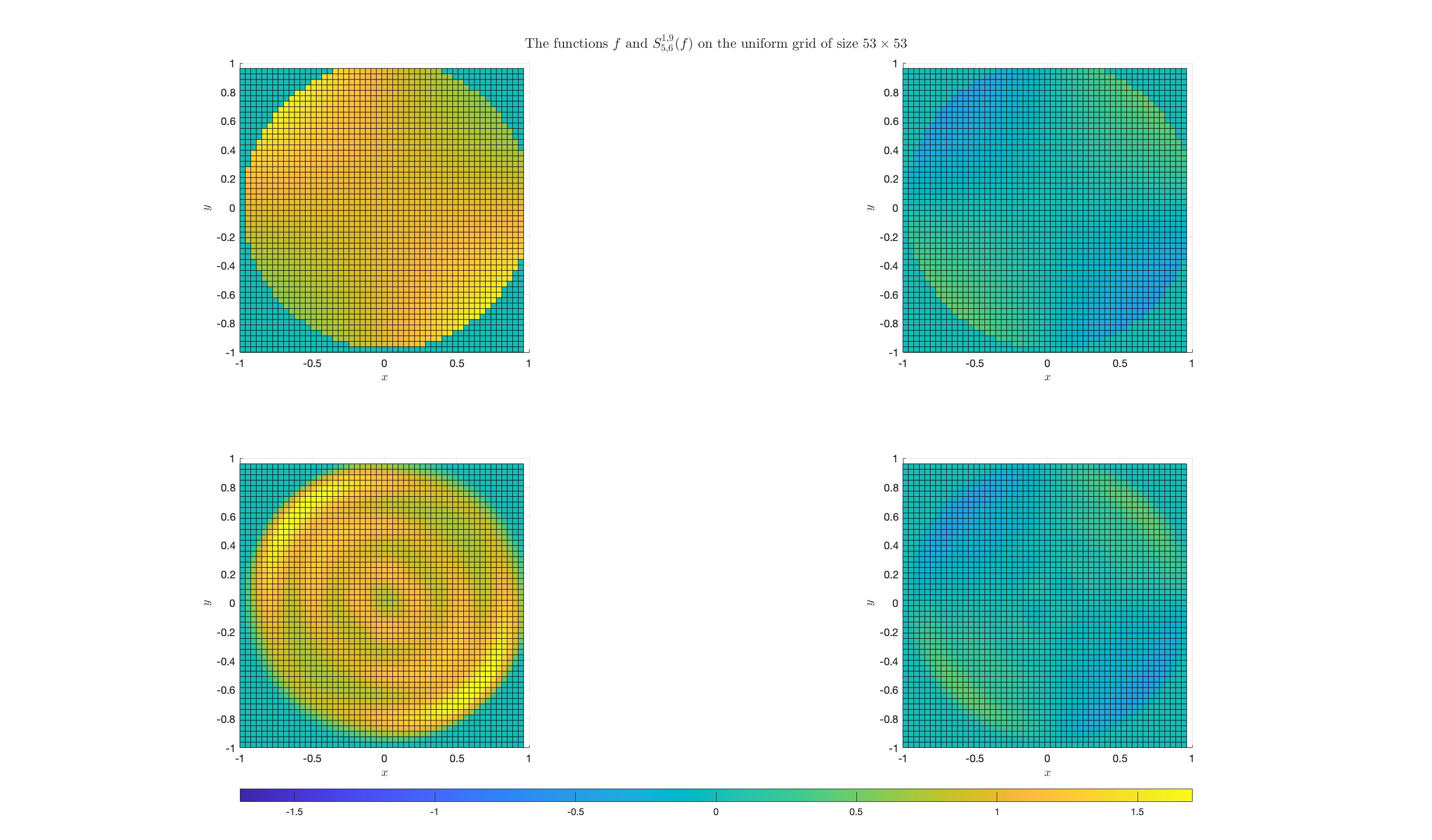}
\caption{The contour plot of $f$ and $S_{5,6}^{9}(f)$  on the uniform grid of size $53\times 53$.}
\label{fig:ExpSinCont}
\end{figure}
\end{example}

The characteristic (indicator) functions of bodies/shapes appear in geometric shape analysis.  Such functions can be canonically considered as the simplest examples of almost everywhere smooth functions supported in disks, if the shape is parametrized by a closed curve.
 
\begin{example}
For $a,b>0$, let $\Omega_{a,b}:=[-a,a]\times[-b,b]$.
Let $\chi_{a,b}$ be the characteristic function of the rectangle $\Omega_{a,b}$. That is 
\[
\chi_{a,b}(x,y):=\left\{\begin{array}{cc}
1 & (x,y)\in\Omega_{a,b} \\ 
0 & (x,y)\not\in\Omega_{a,b}
\end{array}\right.
\]
\begin{figure}[H]
\centering
\includegraphics[keepaspectratio=true,width=\textwidth, height=0.35\textheight]{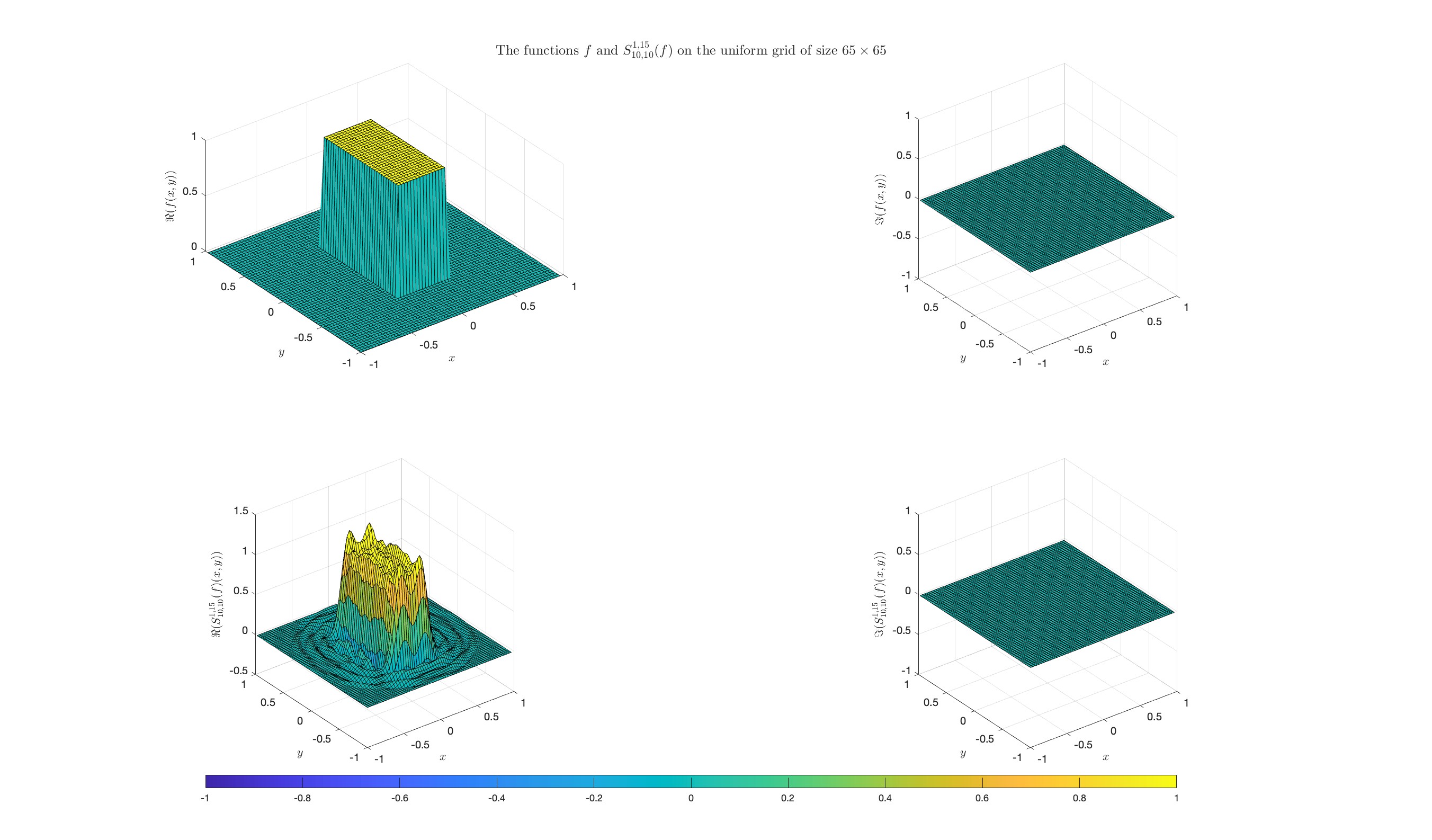}
\caption{The 3D surface plot of $f=\chi_{0.25,0.5}$ and $S_{10,10}^{15}(f)$ on the uniform grid of size $65\times 65$.  It is worth noting that $K[10,10]=15$.}
\label{fig:IOSurf}
\end{figure}

\begin{figure}[H]
\centering
\includegraphics[keepaspectratio=true,width=\textwidth, height=0.35\textheight]{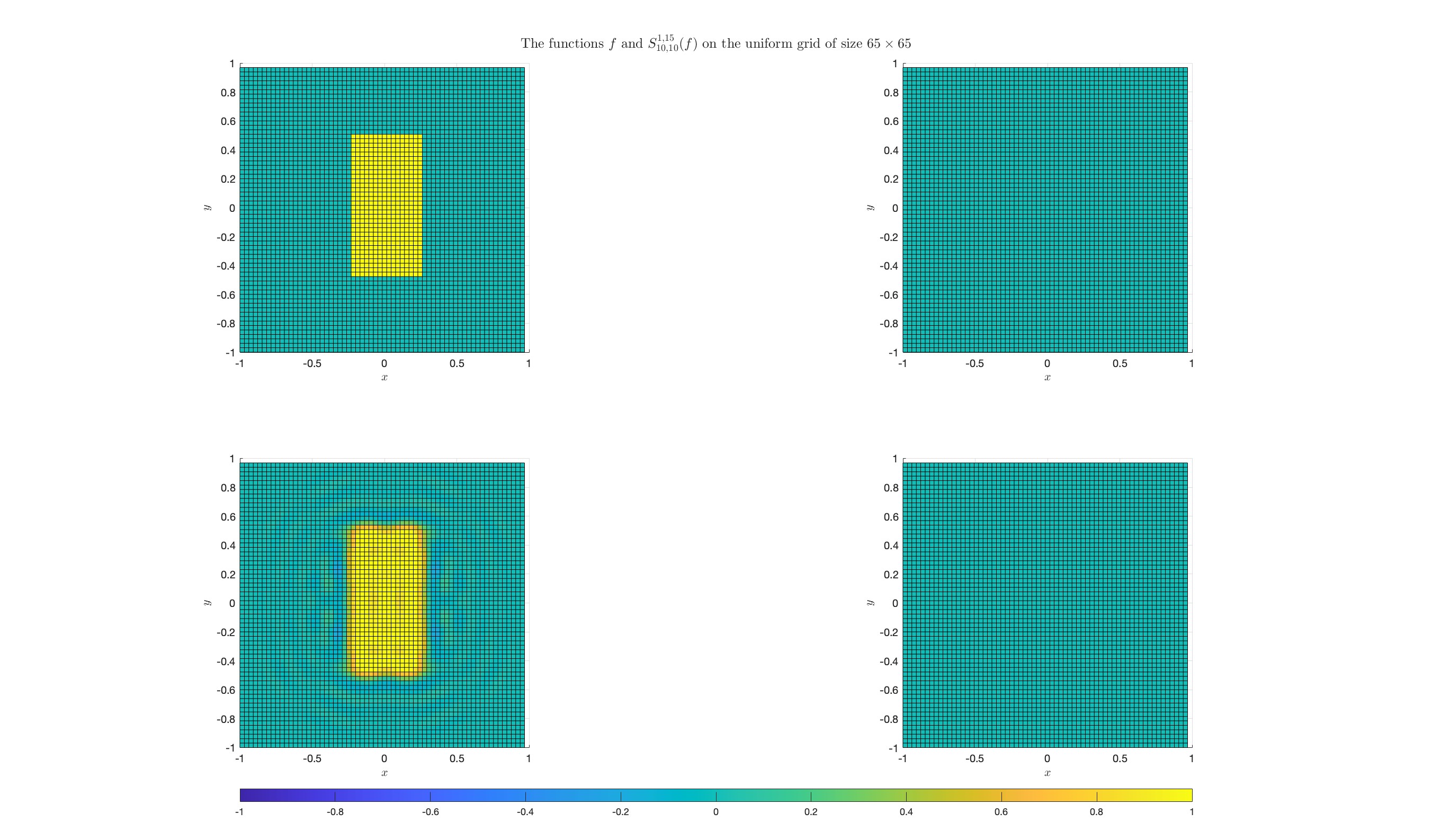}
\caption{The contour plot of $f=\chi_{0.25,0.5}$ and $S_{10,10}^{15}(f)$ on the grid of size $65\times 65$.}
\label{fig:IOCont}
\end{figure}
\end{example}

\newpage 

Next example considers a body inside a non-symmetric two dimensional polygon  given by some ordered sets of vertices. 

\begin{example}
Suppose that $f:=\chi_A$ where $A$ is the body inside the two dimensional polygon given by the set of ordered vertices 
$$\mathcal{V}:=\left\{(0,0), (0,3),(-3,3),(-3,0),(-2,0),(-2,-1),(-1,1),(-1,-2),(2,-2),(2,0)\right\}.$$ 

\begin{figure}[H]
\centering
\includegraphics[keepaspectratio=true,width=\textwidth, height=0.43\textheight]{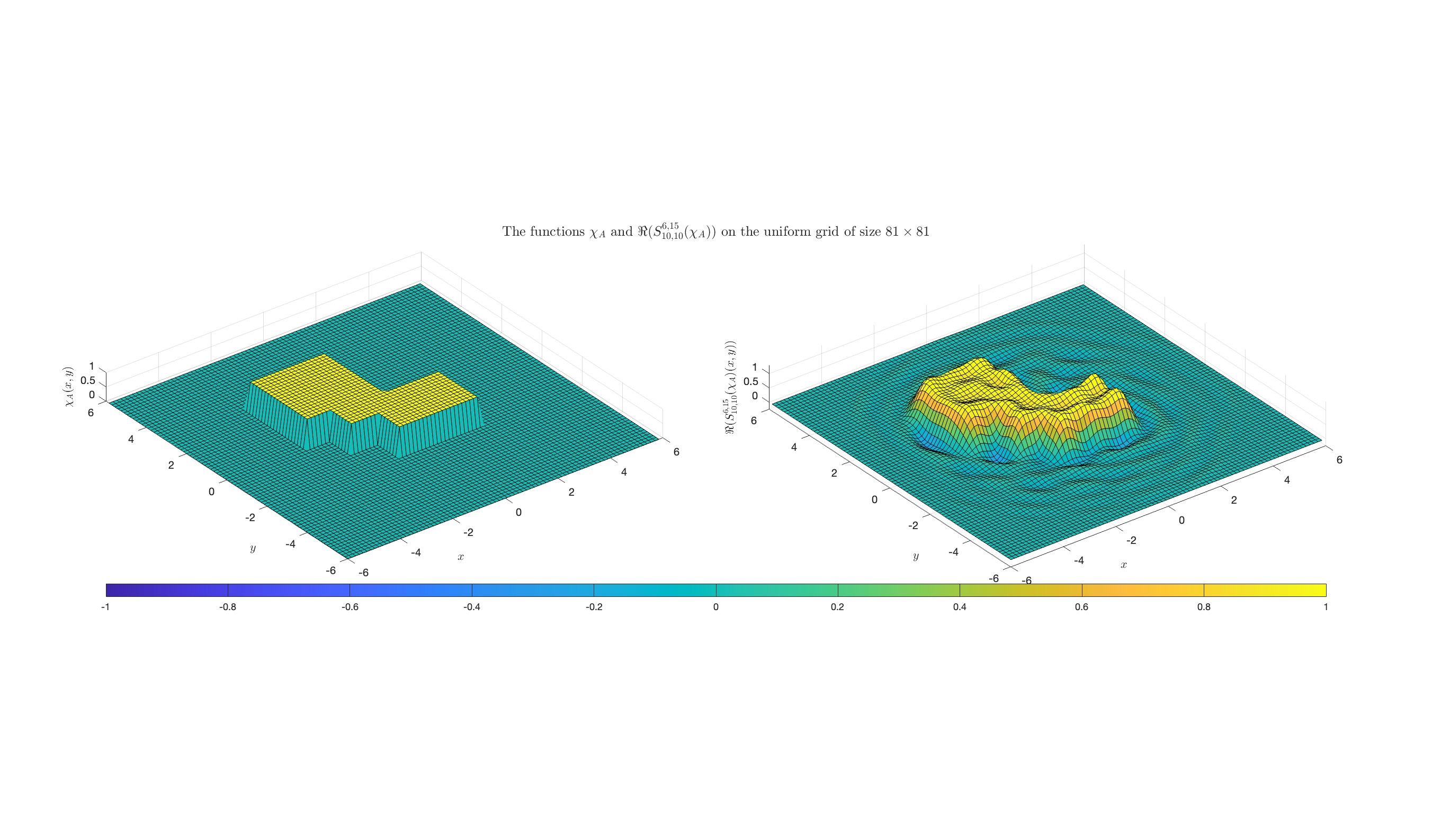}
\caption{The 3D surface plot of $f=\chi_{A}$ and $S_{10,10}^{6,15}(f)$ on the uniform grid of size $81\times 81$.  It is worth noting that $K[10,10]=15$.}
\label{fig:ILSurf}
\end{figure}

\begin{figure}[H]
\centering
\includegraphics[keepaspectratio=true,width=\textwidth, height=0.25\textheight]{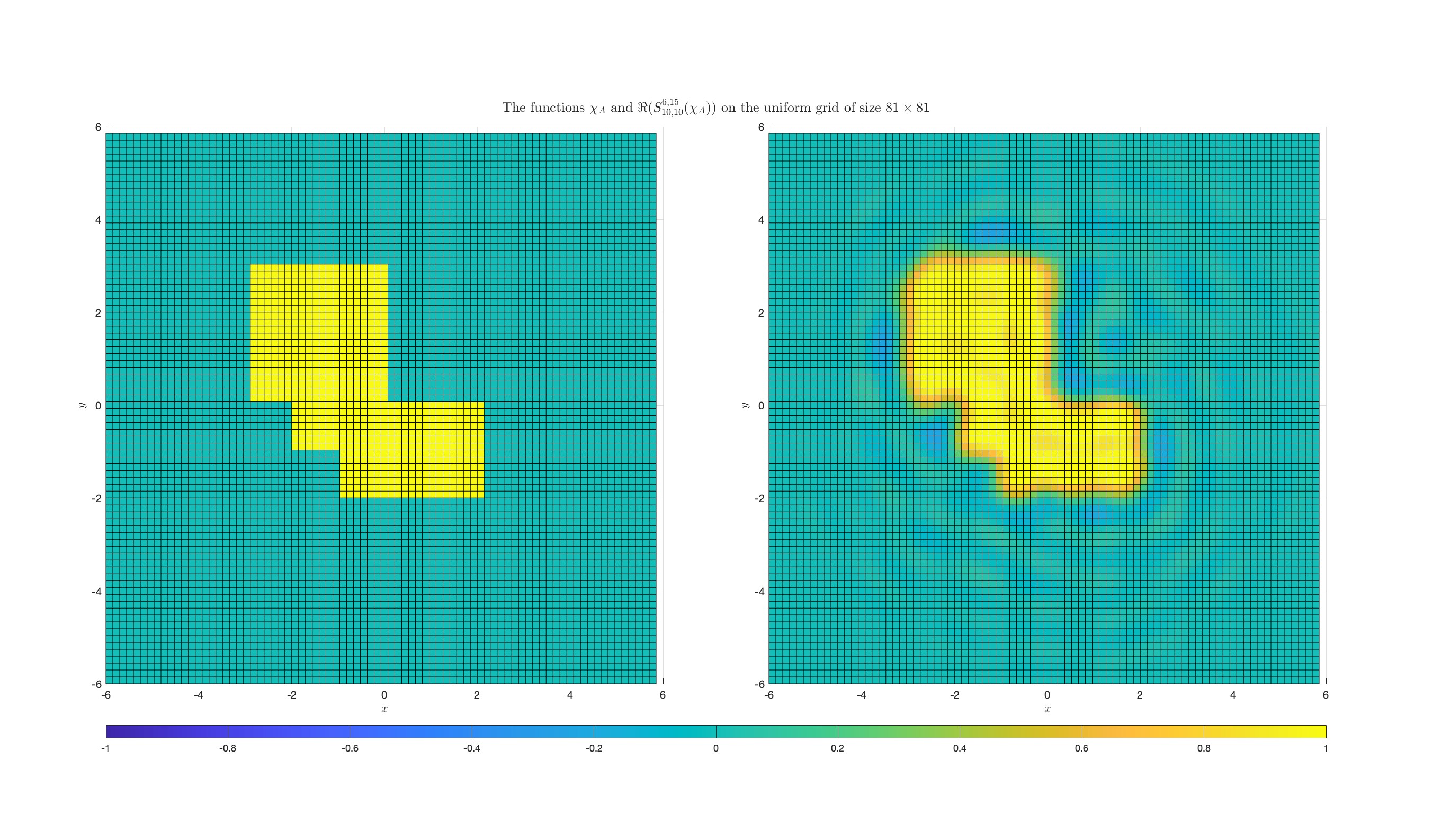}
\caption{The contour plot of $f=\chi_A$ and $S_{10,10}^{1,15}(f)$ on the uniform grid of size $81\times 81$.}
\label{fig:ILCont}
\end{figure}
\end{example}

\newpage
\begin{example}
Let $A$ be the body inside the astroid centred at origin associated to $p:=\frac{2}{3}$ metric, that is $A:=\{\mathbf{x}\in\mathbb{R}^2:\|\mathbf{x}\|_p\le 1\}$.  Suppose that $f:=\chi_A$. 
\begin{figure}[H]
\centering
\includegraphics[keepaspectratio=true,width=\textwidth, height=0.45\textheight]{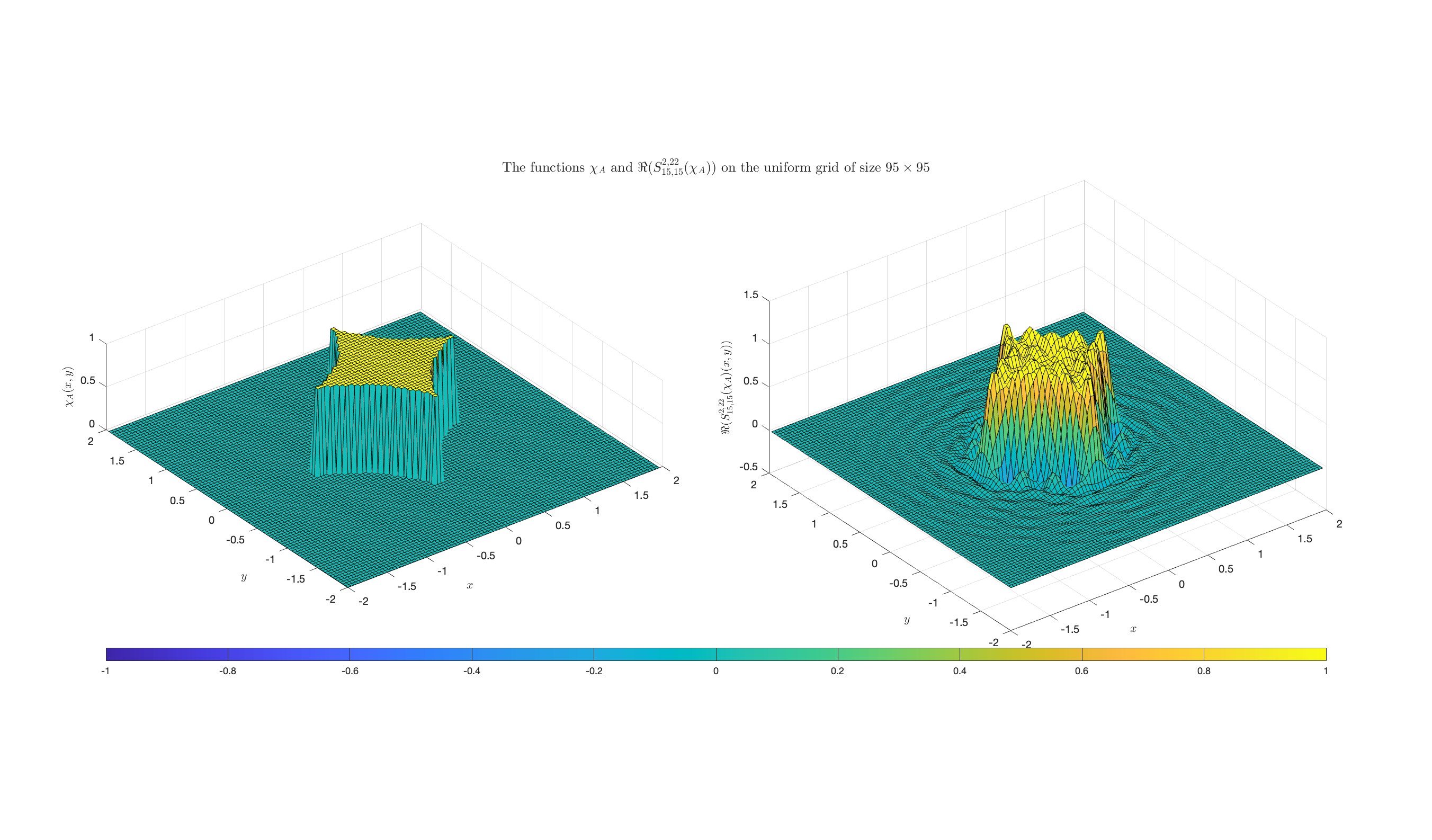}
\caption{The 3D surface plot of $f=\chi_{A}$ and $S_{15,15}^{2,22}(f)$ on the uniform grid of size $95\times 95$.  It is worth noting that $K[15,15]=22$.}
\label{fig:ChSurf}
\end{figure}
\begin{figure}[H]
\centering
\includegraphics[keepaspectratio=true,width=\textwidth, height=0.35\textheight]{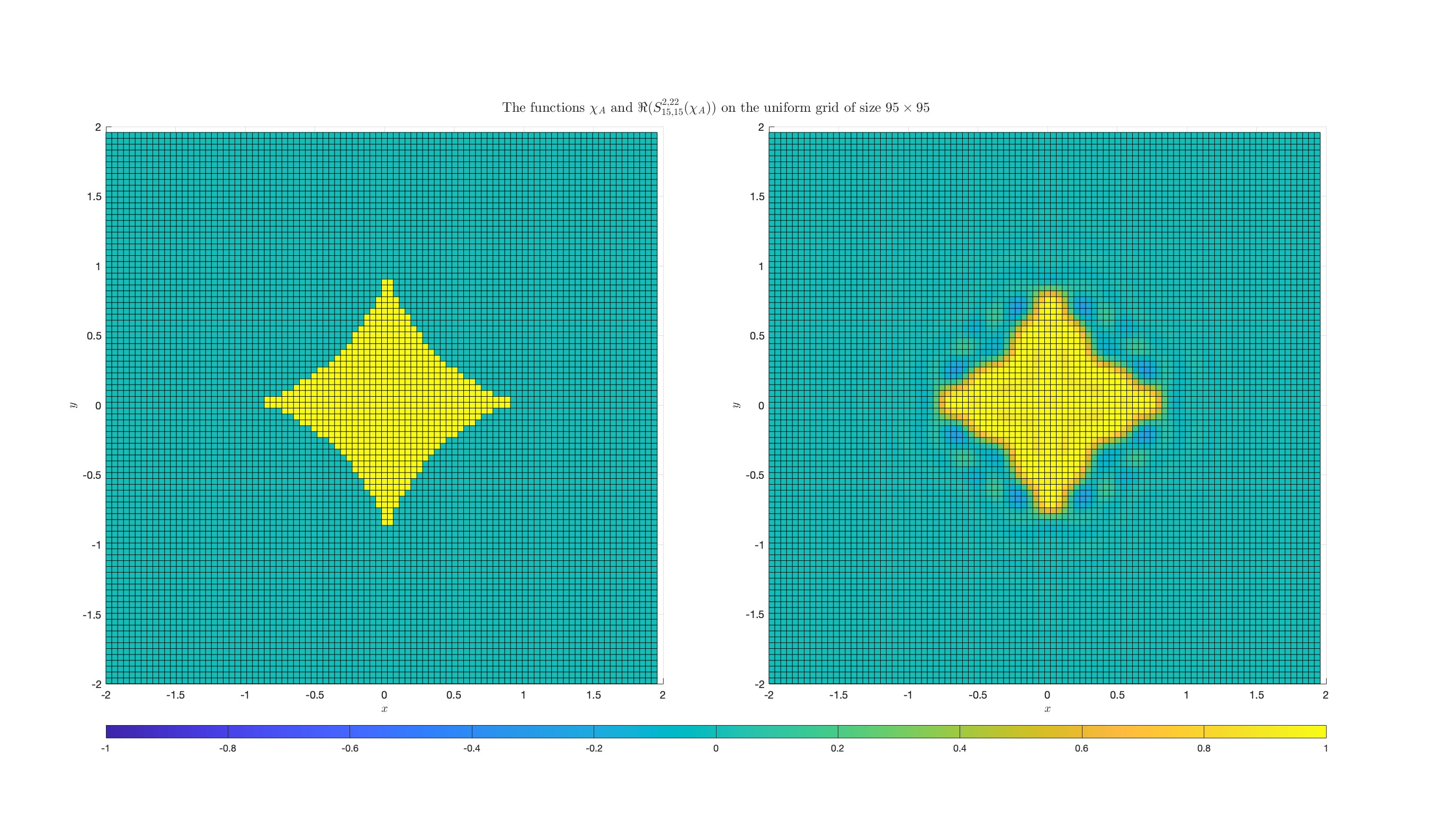}
\caption{The contour plot of $f$ and $S_{15,15}^{2,22}(f)$ on the uniform grid of size $95\times 95$.}
\label{fig:ChCont}
\end{figure}
\end{example}

\section{\bf  FFBT/iFFBT for Convolutions on Disks }

We here develop a unified finite Fourier-Bessel transform for convolutions  supported on disks.   

Let $f,g:\mathbb{R}^2\to\mathbb{C}$ be integrable functions supported in $\mathbb{B}_{1/2}$.  Then $f\ast g$ is supported in $\mathbb{B}$.  Assume $(m,n)\in\mathbb{Z}\times\mathbb{N}$.   Then FFBT of order $K\in\mathbb{N}$ of $C_{m,n}(f\ast g)$,  given by (\ref{DFBT-FFT}), reads as 
\[
C_{m,n}^{K}(f\ast g)=\sum_{\|\mathbf{k}\|_{\infty}\le K}\mathbf{c}(\mathbf{k};m,n)\widehat{f\ast g}(\mathbf{k};2K+1),
\]
with
\[
\widehat{f\ast g}(\mathbf{k};2K+1)=\delta^2\sum_{i=1}^{2K+1}\sum_{j=1}^{2K+1}f\ast g(x_i,x_j)e^{-\pi\ii (k_1x_i+k_2x_j)},
\]
where $x_j:=-1+(j-1)\delta$, $\delta:=\frac{2}{L}$, $L:=2K+1$,  and $\mathbf{c}(\mathbf{k};m,n)$ given by (\ref{cakmn}) for every $\mathbf{k}:=(k_1,k_2)^T$.  

The later requires values of the convolution integral $f\ast g$ on the uniform sampling grid of size $L\times L$.  This imposes several practical disadvantages/inflexibilities to the FFBT/iFFBT numerical algorithms when applied to convolutions.  To begin with,   it makes the algorithms computationally expensive.  In addition,  the numerical scheme is less flexible/attractive to problems which the ultimate goal is computing/approximating values of the convolutions. 

To improve this,  we develop the unified FFBT for convolution of functions $f,g\in\mathcal{V}_{1/2}$,  denoted by $C_{m,n}^{K}[f,g]$,  as an accurate and numerical approximation of $C_{m,n}^{K}(f\ast g)$ which can be computed independent of the values of the convolution integral itself.  The guiding idea is to approximate the finite Fourier transform $
\widehat{f\ast g}(\mathbf{k};L),$
with the multiplication of finite Fourier transforms 
$\widehat{f}(\mathbf{k};L)\widehat{g}(\mathbf{k};L)$.

Initially,  we prove that multiplication of finite Fourier transform of functions can be considered as an accurate approximation for finite Fourier transform of their convolutions. 

\begin{proposition}\label{fftFGfftFfftG}
{\it Let $f,g\in\mathcal{C}^1(\mathbb{R}^2)$ be supported in $\mathbb{B}_{1/2}^\circ$.  Suppose $\mathbf{k}\in\mathbb{Z}^2$ and $K\in\mathbb{N}$.  Then there exists a universal constant $\kappa^\ast>0$ and constant $c_{f,g}>0$ such that 
\[
\left|\widehat{f\ast g}(\mathbf{k};2K+1)-\widehat{f}(\mathbf{k};2K+1)\widehat{g}(\mathbf{k};2K+1)\right|\le\frac{\kappa^\ast c_{f,g}}{K},\hspace{1cm}{\rm if}\ \|\mathbf{k}\|_\infty\le K.
\]
}\end{proposition}
\begin{proof}
Let $c_{f,g}:=\max\{\|f\|_1\|\nabla g\|_\infty,4\|g\|_\infty\|\nabla f\|_\infty\}$,   $L:=2K+1$ and $\|\mathbf{k}\|_\infty\le K$. 
Proposition \ref{mainXIb} gets 
 \begin{align*}
&\left|\widehat{f}(\mathbf{k})\widehat{g}(\mathbf{k})-\widehat{f}(\mathbf{k};L)\widehat{g}(\mathbf{k};L)\right|
\le\left|\widehat{f}(\mathbf{k})\widehat{g}(\mathbf{k})-\widehat{f}(\mathbf{k})\widehat{g}(\mathbf{k};L)\right|
 +\left|\widehat{f}(\mathbf{k})\widehat{g}(\mathbf{k};L)-\widehat{f}(\mathbf{k};L)\widehat{g}(\mathbf{k};L)\right|
\\&=|\widehat{f}(\mathbf{k})||\widehat{g}(\mathbf{k})-\widehat{g}(\mathbf{k};L)|
+|\widehat{g}(\mathbf{k};L)||\widehat{f}(\mathbf{k})-\widehat{f}(\mathbf{k};L)|
\le \|f\|_1\left|\widehat{g}(\mathbf{k})-\widehat{g}(\mathbf{k};L)\right|
+4\|g\|_\infty|\widehat{f}(\mathbf{k})-\widehat{f}(\mathbf{k};L)|.
\end{align*}
Proposition 8.10 of \cite{Foll.R} implies that $f\ast g\in \mathcal{C}^1(\mathbb{R}^2)$ with $\|\nabla(f\ast g)\|_\infty\le \|f\|_1\|\nabla g\|_\infty$.
So,  we obtain  
\begin{align*}
&\left|\widehat{f\ast g}(\mathbf{k};L)-\widehat{f}(\mathbf{k};L)\widehat{g}(\mathbf{k};L)\right|
\le \left|\widehat{f\ast g}(\mathbf{k};L)-\widehat{f\ast g}(\mathbf{k})\right|+\left|\widehat{f\ast g}(\mathbf{k})-\widehat{f}(\mathbf{k};L)\widehat{g}(\mathbf{k};L)\right|
\\&\le \frac{\kappa\|\nabla(f\ast g)\|_\infty}{K}+\left|\widehat{f}(\mathbf{k})\widehat{g}(\mathbf{k})-\widehat{f}(\mathbf{k};L)\widehat{g}(\mathbf{k};L)\right|
\le\frac{\kappa \|f\|_1\|\nabla g\|_\infty}{K}+\frac{2\kappa c_{f,g}}{K}\le\frac{3\kappa c_{f,g}}{K}.
\end{align*}
\end{proof}
\begin{remark}
In terms of approximation theory,  Proposition \ref{fftFGfftFfftG},  reads as  
\[
\left|\widehat{f\ast g}(\mathbf{k};2K+1)-\widehat{f}(\mathbf{k};2K+1)\widehat{g}(\mathbf{k};2K+1)\right|=\mathcal{O}(1/K),
\hspace{1cm}{\rm if}\ \|\mathbf{k}\|_\infty\le K.
\]
\end{remark}
\subsection{Finite Fourier-Bessel transform of convolutions}
Throughout, we discuss the unified finite Fourier-Bessel transform (FFBT) of compactly supported convolutions on disks using finite Fourier transform.
Suppose that $f,g:\mathbb{R}^2\to\mathbb{C}$ are functions supported in $\mathbb{B}_{1/2}$.
We then define  
\begin{equation}\label{CmnKFG}
C_{m,n}^{K}[f,g]:=\sum_{\|\mathbf{k}\|_\infty\le K}\mathbf{c}(\mathbf{k};m,n)\widehat{f}(\mathbf{k};2K+1)\widehat{g}(\mathbf{k};2K+1),
\end{equation}
where $\mathbf{c}(\mathbf{k};m,n)$ given by (\ref{cakmn}) and $\widehat{f}(\mathbf{k};2K+1)$ (resp.  $\widehat{g}(\mathbf{k};2K+1)$) is given by (\ref{Xi.kL}), for $\mathbf{k}:=(k_1,k_2)^T$.

Next theorem shows that $C_{m,n}^{K}[f,g]$ is an accurate approximation of $C_{m,n}(f\ast g)$ if $K\ge K_{m,n}$.
\begin{theorem}\label{CmnKFG-F*G}
Let $(m,n)\in\mathbb{Z}\times\mathbb{N}$ and $f,g\in\mathcal{C}^1(\mathbb{R}^2)$ be supported in $\mathbb{B}_{1/2}^\circ$.  Then there exist  constants  $\gamma_{m,n},d_{f,g}>0$ such that 
\[
\left|C_{m,n}^{K}[f,g]-C_{m,n}(f\ast g)\right|\le\frac{2\gamma_{m,n}d_{f,g}}{K},\hspace{1.5cm}{\rm for}\ K\ge K_{m,n}.
\]
\end{theorem}
\begin{proof}
Let $L:=2K+1$ and $\mathbf{c}(\mathbf{k}):=\mathbf{c}(\mathbf{k};m,n)$.  Invoking Proposition \ref{fftFGfftFfftG},  we obtain 
\begin{align*}
&\left|C_{m,n}^{K}[f,g]-C_{m,n}^{K}(f\ast g)\right|
=\left|\sum_{\|\mathbf{k}\|_\infty\le K}\mathbf{c}(\mathbf{k})\left(\widehat{f}(\mathbf{k};L)\widehat{g}(\mathbf{k};L)-\widehat{f\ast g}(\mathbf{k};L)\right)\right|
\\&\le\sum_{\|\mathbf{k}\|_\infty\le K}|\mathbf{c}(\mathbf{k})||\widehat{f}(\mathbf{k};L)\widehat{g}(\mathbf{k};L)-\widehat{f\ast g}(\mathbf{k};L)|
\le\frac{\kappa^* c_{f,g}}{K}\sum_{\|\mathbf{k}\|_\infty\le K}|\mathbf{c}(\mathbf{k})|\le \frac{\kappa^*c_{f,g}\gamma_{m,n}}{K}.
\end{align*}
Suppose $d_{f,g}:=\max\{\kappa^*c_{f,g},2c_{f\ast g}\}$ and $K\ge K_{m,n}$.  Then applying Theorem \ref{CmnKerr} for $f\ast g$, we obtain 
\begin{align*}
\left|C_{m,n}^{K}[f,g]-C_{m,n}(f\ast g)\right|
&\le \left|C_{m,n}^{K}[f,g]-C_{m,n}^K(f\ast g)\right|+\left|C_{m,n}^K(f\ast g)-C_{m,n}(f\ast g)\right|
\le\frac{2\gamma_{m,n}d_{f,g}}{K}.
\end{align*}
\end{proof}
\begin{corollary}
{\it Let $(m,n)\in\mathbb{Z}\times\mathbb{N}$ and $f,g\in\mathcal{C}^1(\mathbb{R}^2)$ be supported in $\mathbb{B}_{1/2}^\circ$.  Then 
\[
\left|C_{m,n}^{K}[f,g]-C_{m,n}(f\ast g)\right|=\mathcal{O}(1/K),
\hspace{1.5cm}{\rm for}\ K\ge K_{m,n}.
\]
}\end{corollary}
The following diagram compares construction of Fourier-Bessel transform $C_{m,n}(f\ast g)$,  and the unified finite Fourier-Bessel transform $C_{m,n}^K[f,g]$ for $f,g\in\mathcal{V}_{1/2}$ at order $K\in\mathbb{N}$. 
\begin{figure}[H]
\begin{tikzcd}
f\ast g \arrow[r, "\widehat{}"] \arrow[d, "C_{m,n}"] &  \left(\widehat{f}(\mathbf{k})\widehat{g}(\mathbf{k})\right)_{\mathbf{k}\in\mathbb{Z}^2} \arrow[d] \arrow[r, "\mathcal{O}",equal] & \left(\widehat{f}(\mathbf{k};2K+1)\widehat{g}(\mathbf{k};2K+1)\right)_{\|\mathbf{k}\|_\infty\le K}\arrow[d]\\
C_{m,n}(f\ast g) \arrow[r,  "(\ref{C.ZBVC})",equal]& \sum_{\mathbf{k}\in\mathbb{Z}^2}\mathbf{c}(\mathbf{k};m,n)\widehat{f}(\mathbf{k})\widehat{g}(\mathbf{k})
 \arrow[r,  "\mathcal{O}",equal] & C_{m,n}^K[f,g]
\end{tikzcd}
\caption{The diagram comparing construction of Fourier-Bessel transforms vs.  FFBT}
\end{figure}

\subsection{Inverse Finite Fourier-Bessel transform of convolutions}
Then we introduce the inverse unified finite Fourier-Bessel transforms for convolutions using finite Fourier transform. 

Suppose $f,g:\mathbb{R}^2\to\mathbb{C}$ are supported in $\mathbb{B}_{1/2}$ and $M,N,K\in\mathbb{N}$.   
We then define 
\begin{align}\label{SMNKFG-FFT}
S_{M,N}^{K}[f,g](x,y)
:=\sum_{m=-M}^M\sum_{n=1}^NC_{m,n}^{K}[f,g]\Psi_{m,n}(x,y),
\end{align}
for $(x,y)\in\mathbb{R}^2$, where $C_{m,n}^{K}[f,g]$ is given by (\ref{CmnKFG}). 

Next we show that $S_{M,N}^{K}[f,g]$ is an accurate approximation of $S_{M,N}(f\ast g)$.

\begin{theorem}
Let $f,g\in\mathcal{C}^1(\mathbb{R}^2)$ be supported in $\mathbb{B}_{1/2}^\circ$ and $M,N\in\mathbb{N}$.   Then 
\[
\left\|S_{M,N}^{K}[f,g]-S_{M,N}(f\ast g)\right\|_\infty\le\frac{2d_{f,g}D[M,N]}{K}\hspace{1.5cm}{\rm for}\ K\ge K[M,N].
\]
\end{theorem}
\begin{proof}
Let $K[M,N]\le K$.  Then $K_{m,n}\le K$ for $-M\le m\le M$ and $1\le n\le N$.  So 
Theorem \ref{CmnKFG-F*G} implies  
\begin{equation*}
\left|C_{m,n}^{K}[f,g]-C_{m,n}^{K}(f\ast g)\right|\le\frac{2\gamma_{m,n}d_{f,g}}{K}.
\end{equation*}
Therefore,  we achieve 
\begin{align*}
\left|S_{M,N}^{K}[f,g](x,y)-S_{M,N}(f\ast g)(x,y)\right|
&\le\hspace{-0.3cm}\sum_{m=-M}^M\sum_{n=1}^N\left|C_{m,n}^{K}[f,g]-C_{m,n}(f\ast g)\right|\left|\Psi_{m,n}(x,y)\right|
\le\hspace{-0.1cm} \frac{2d_{f,g}D[M,N]}{K},
\end{align*}
which implies that $
\left\|S_{M,N}^{K}[f,g]-S_{M,N}(f\ast g)\right\|_\infty\le\frac{2d_{f,g}D[M,N]}{K}$.
\end{proof}
\begin{corollary}
{\it Let $f,g\in\mathcal{C}^1(\mathbb{R}^2)$ be supported in $\mathbb{B}_{1/2}^\circ$ and $M,N\in\mathbb{N}$.   Then
\[
\left\|S_{M,N}^{K}[f,g]-S_{M,N}(f\ast g)\right\|_\infty=\mathcal{O}(1/K), \hspace{1.5cm}{\rm for}\ K\ge K[M,N].
\]
}\end{corollary}
\begin{remark}\label{BaCNV}
As discussed in Remark \ref{Ba},  the unified iFFBT numerical scheme for convolutions can also be applied for approximation of convolutions  supported in arbitrary disks by scaling.  Let $a>0$ and $f,g:\mathbb{R}^2\to\mathbb{C}$ be integrable functions supported in $\mathbb{B}_{a/2}$.  Then $\widetilde{f},\widetilde{g}$ are supported in $\mathbb{B}_{1/2}$.  Assume $S_{M,N}(\widetilde{f}\ast \widetilde{g})$ is an approximation of $\widetilde{f}\ast\widetilde{g}$ in $L^2$-sense for large enough $M,N\in\mathbb{N}$. Then 
$$S_{M,N}^{a,K}[f,g](\mathbf{x}):=S_{M,N}^K[\widetilde{f},\widetilde{g}](a^{-1}\mathbf{x}),$$ is a numerical approximation of $(f\ast g)(\mathbf{x})$ in $L^2$-sense, if $K\ge K[M,N]$.
\end{remark}
\section{\bf Numerical Convolutions using iFFBT/FFBT in MATLAB}
This section presents matrix forms for FFBT/iFFBT of convolutions including some numerical experiments in MATLAB. 
\subsection{Matrix form of FFBT for convolutions on disks}
Let $f,g\in\mathcal{V}_{1/2}$ and $(m,n)\in\mathbb{Z}\times\mathbb{N}$.  Assume $K\in\mathbb{N}$ and $L:=2K+1$.  Suppose  $x_j:=-1+(j-1)\delta$ with $\delta:=\frac{2}{L}$ and $1\le j\le L$.  Let $\mathbf{F},\mathbf{G}\in\mathbb{C}^{L\times L}$ be given by $\mathbf{F}(i,j):=f(x_i,x_j)$ and $\mathbf{G}(i,j):=g(x_i,x_j)$ for every $1\le i,j\le L$.
Then
\begin{align*}
&C_{m,n}^{K}[f,g]
=\sum_{\|\mathbf{k}\|_\infty\le K}\mathbf{c}(\mathbf{k};m,n)\widehat{f}(\mathbf{k};L)\widehat{g}(\mathbf{k};L)
\\&=\delta^4\sum_{k_1=-K}^K\sum_{k_2=-K}^K\mathbf{c}(k_1,k_2;m,n)\widehat{\mathbf{F}}\odot\widehat{\mathbf{G}}(\tau_L(k_1)+1,\tau_L(k_2)+1)
=\delta^4\sum_{l=1}^{L}\sum_{\ell=1}^{L}\mathbf{Q}^\times(m,n;\ell,l)\widehat{\mathbf{F}}\odot\widehat{\mathbf{G}}(l,\ell),
\end{align*}
where $\odot$ is the Hadamard product,  $\widehat{\mathbf{F}}$ (resp. $\widehat{\mathbf{G}}$) is the DFT of the matrices $\mathbf{F}$ (resp. $\mathbf{G}$) and the matrix $\mathbf{Q}^\times(m,n)\in\mathbb{C}^{L\times L}$ is given by 
\begin{equation}
\mathbf{Q}^\times(m,n)_{\ell,l}=\left\{\begin{array}{llll}
\mathbf{c}(l-1,\ell-1;m,n) & {\rm if}\ 1\le \ell\le K+1,\ 1\le l\le K+1\\
\mathbf{c}(l-L-1,\ell-1;m,n)& {\rm if}\ 1\le \ell\le K+1,\ K+2\le l\le 2K+1\\
\mathbf{c}(l-1,\ell-L-1;m,n)  & {\rm if}\ K+2\le \ell\le 2K+1,  \ 1\le l\le K+1\\
\mathbf{c}(l-L-1,\ell-L-1;m,n) & {\rm if}\ K+2\le \ell\le L,\ K+2\le l\le L\\
\end{array}\right.
\end{equation} 
Therefore, we achieve 
\begin{align}\label{CmnKFG-FFT-TR}
C_{m,n}^{K}[f,g]=\delta^4\mathrm{tr}\left[\mathbf{Q}^\times(m,n)\left(\widehat{\mathbf{F}}\odot\widehat{\mathbf{G}}\right)\right].
\end{align} 

The following Algorithm \ref{CmnKFG-FFTalg} approximates the Fourier-Bessel coefficient $C_{m,n}(f\ast g)$ using $C_{m,n}^{K}[f,g]$ for functions $f,g:\mathbb{R}^2\to\mathbb{C}$ supported on $\mathbb{B}_{1/2}$ with respect to a given absolute error.
\begin{algorithm}[H]
\caption{Finding $\epsilon$-approximation of $C_{m,n}(f\ast g)$ using $C_{m,n}^{K}[f,g]$} 
\begin{algorithmic}[1]
\State{\bf input data}
The functions $f,g:\mathbb{R}^2\to\mathbb{C}$ supported on $\mathbb{B}_{1/2}$,  given error $\epsilon>0$,  and $(m,n)\in\mathbb{Z}\times\mathbb{N}$
\State {\bf output result} $C_{m,n}^{K}[f,g]$ with the absolute error $\le\epsilon$\\ 
Let $K_{m,n}:=\lceil\frac{z_{m,n}}{\pi}\rceil$,  and $\beta_\epsilon:=\max\{2\epsilon^{-1}d_{f,g}\gamma_{m,n},K_{m,n}\}$.\\
	Find $K\in\mathbb{N}$ such that $K\ge\beta_\epsilon$ and let $L:=2K+1$.\\
Load precomputed data $\mathbf{Q}^\times(m,n)\in\mathbb{C}^{L\times L}$\\
Generate the sampling grid $(x_i,x_j)$ with $x_i:=-1+(i-1)\delta$, and $\delta:=\frac{2}{L}$,  for $1\le i\le L$\\
Generate sampled values $\mathbf{F}:=(f(x_i,x_j))_{1\le i,j\le L}$ and  $\mathbf{G}:=(g(x_i,x_j))_{1\le i,j\le L}$.\\
	Compute the DFT matrices $\widehat{\mathbf{F}}$ and $\widehat{\mathbf{G}}$.\\
	Compute the Hadamard product 
		$\widehat{\mathbf{F}}\odot\widehat{\mathbf{G}}:=\widehat{\mathbf{F}}.*\widehat{\mathbf{G}}.$
\State Compute 
$C_{m,n}^{K}[f,g]:=\delta^4\mathrm{tr}\left[\mathbf{Q}^\times(m,n)\widehat{\mathbf{F}}\odot\widehat{\mathbf{G}}\right].$
\end{algorithmic} 
\label{CmnKFG-FFTalg}
\end{algorithm}
\subsection{Matrix form of iFFBT for convolutions on disks}
Let $f,g\in\mathcal{V}_{1/2}$.  Suppose $M,N,K\in\mathbb{N}$ and $(x,y)\in\mathbb{R}^2$.  Let $L:=2K+1$.  Then using (\ref{SMNKFG-FFT}) and (\ref{CmnKFG-FFT-TR}), we get 
\begin{align*}
S_{M,N}^{K}[f,g](x,y)
&=\sum_{m=-M}^M\sum_{n=1}^N\left(\delta^4\sum_{l=1}^{L}\sum_{\ell=1}^{L}\mathbf{Q}^\times(m,n)_{\ell,l}\widehat{\mathbf{F}}(l,\ell)\widehat{\mathbf{G}}(l,\ell)\right)\Psi_{m,n}(x,y)
\\&=\delta^4\sum_{l=1}^{L}\sum_{\ell=1}^{L}\left(\sum_{m=-M}^M\sum_{n=1}^N\mathbf{Q}^\times(m,n)_{\ell,l}\Psi_{m,n}(x,y)\right)\widehat{\mathbf{F}}(l,\ell)\widehat{\mathbf{G}}(l,\ell).
\end{align*}
Suppose that $1\le l,\ell\le L$ are given.  Then 
\begin{align*}
\sum_{m=-M}^M\sum_{n=1}^N\mathbf{Q}^\times(m,n)_{l,\ell}\Psi_{m,n}(x,y)
=\mathrm{tr}\left[\mathbf{H}^\times(l,\ell,:,:)\mathbf{P}(x,y)^T\right],
\end{align*}
where $\mathbf{H}^\times(l,\ell)\in\mathbb{C}^{(2M+1)\times N}$ is given by 
\begin{equation}\label{H*mat}
\mathbf{H}^\times(l,\ell)_{t,n}=\mathbf{H}^\times(l,\ell,t,n):=
\mathbf{Q}^\times(t-M-1,n,l,\ell),
\end{equation}
for every $1\le l,\ell\le L$,  $1\le t\le 2M+1$, and $1\le n\le N$.
Then, we have 
\begin{align*}
S_{M,N}^K[f,g](x,y)
&=\delta^4\sum_{l=1}^{L}\sum_{\ell=1}^{L}\left(\sum_{m=-M}^M\sum_{n=1}^N\mathbf{Q}^\times(m,n)_{\ell,l}\Psi_{m,n}(x,y)\right)\widehat{\mathbf{F}}(l,\ell)\widehat{\mathbf{G}}(l,\ell)
\\&=\delta^4\sum_{l=1}^{L}\sum_{\ell=1}^{L}\mathrm{tr}\left[\mathbf{H}^\times(l,\ell)\mathbf{P}(x,y)^T\right]\widehat{\mathbf{F}}(l,\ell)\widehat{\mathbf{G}}(l,\ell)
=\delta^4\mathrm{tr}\left[\mathbf{K}^\times(x,y)\left(\widehat{\mathbf{F}}\odot\widehat{\mathbf{G}}\right)\right],
\end{align*}
implying that 
\begin{equation}\label{SMNKFG.tr.xy}
S_{M,N}^{K}[f,g](x,y)=\delta^4\mathrm{tr}\left[\mathbf{K}^\times(x,y)\left(\widehat{\mathbf{F}}\odot\widehat{\mathbf{G}}\right)\right],
\end{equation}
where the convolution kernel $\mathbf{K}^\times(x,y)\in\mathbb{C}^{L\times L}$ is defined by 
\[
\mathbf{K}^\times(x,y)_{\ell,l}:=\mathrm{tr}\left[\mathbf{H}^\times(l,\ell,:,:)\mathbf{P}(x,y)^T\right],\hspace{1cm}\ {\rm for}\ 1\le l,\ell\le L.
\]

The Algorithm \ref{SMNKFG-FFTalg} approximates the Fourier-Bessel partial sum $S_{M,N}(f\ast g)$ of the convolution $f\ast g$ with respect to a given absolute error $\epsilon$ using $S_{M,N}^{K}[f,g]$.

\begin{algorithm}[H]
\caption{Finding $\epsilon$-approximation of $S_{M,N}(f\ast g)(x,y)$ using $S_{M,N}^{K}[f,g](x,y)$} 
\begin{algorithmic}[1]
\State{\bf input data} functions $f,g\in\mathcal{C}^1(\mathbb{R}^2)$ supported in $\mathbb{B}_{1/2}^\circ$,  error $\epsilon>0$,  orders $M,N\in\mathbb{N}$, and $(x,y)\in\Omega$
\State {\bf output result} $S_{M,N}^{K}[f,g](x,y)$ with the absolute error $\le\epsilon$
\State Compute $D[M,N]:=\frac{1}{\sqrt{2}}\sum_{m=-M}^M\sum_{n=1}^N\frac{\gamma_{m,n}}{|J_{m+1}(z_{m,n})|}.$
\State Compute $K[M,N]:=\max\{K_{m,n}:0\le m\le M,1\le n\le N\}.$
\State Put $\beta_\epsilon:=\max\{2\epsilon^{-1}d_{f,g}D[M,N],K[M,N]\}$,.
\State Find $K\in\mathbb{N}$ such that $K\ge\beta_\epsilon$ and let $L:=2K+1$.
\State Load the precomputed values $\mathbf{P}(x,y)\in\mathbb{C}^{(2M+1)\times N}$ and $\mathbf{H}^\times\in\mathbb{C}^{L\times L\times(2M+1)\times N}$
\State Generate the convolution kernel $\mathbf{K}^\times(x,y)\in\mathbb{C}^{L\times L}$ using 
\[
\mathbf{K}^\times(x,y)_{\ell,l}:=\mathrm{tr}\left[\mathbf{H}(l,\ell,:,:)\mathbf{P}(x,y)^T\right],\hspace{1cm}\ {\rm for}\ 1\le l,\ell\le L.
\]
\State Generate the sampling grid $(x_i,x_j)$ with $x_i:=-1+(i-1)\delta$, and $\delta:=\frac{2}{L}$,  for $1\le i\le L$\\
Generate sampled values $\mathbf{F}:=(f(x_i,x_j))_{1\le i,j\le L}$ and  $\mathbf{G}:=(g(x_i,x_j))_{1\le i,j\le L}$)\\
	Compute the matrices $\widehat{\mathbf{F}}$ and $\widehat{\mathbf{G}}$.\\
	Compute the Hadamard product 
		$\widehat{\mathbf{F}}\odot\widehat{\mathbf{G}}:=\widehat{\mathbf{F}}.*\widehat{\mathbf{G}}.$
\State Compute 
$S_{M,N}^{K}[f,g](x,y):=\delta^4\mathrm{tr}\left[\mathbf{K}^\times(x,y)\left(\widehat{\mathbf{F}}\odot\widehat{\mathbf{G}}\right)\right].$
\end{algorithmic} 
\label{SMNKFG-FFTalg}
\end{algorithm}

\subsection{Numerical experiments} 
We here implemented experiments in MATLAB for convolution of functions supported on disks using the discussed matrix approach (\ref{SMNKFG.tr.xy}).  
It is worth mentioning that the matrix forms (\ref{CmnKFG-FFT-TR}) and (\ref{SMNKFG.tr.xy}) have several benefits in computing approximations of convolutions supported on disks.  To begin with,  the forms reduced multivariate summations  in the right hand-side of (\ref{CmnKFG}) and (\ref{SMNKFG-FFT}) into matrix trace/multiplication which can be computed using fast matrix multiplications algorithms.  Also,  the matrix forms (\ref{CmnKFG-FFT-TR}) and (\ref{SMNKFG.tr.xy}) involves the DFT of sampled values matrices which can be performed using FFT. 

Let $b>0$ and $\chi_{\mathbb{B}_b}$ be the characteristic function of the disk of radius $b$ . 

\begin{example}
Suppose that $f=g=\chi_{\mathbb{B}_1}$.
\begin{figure}[H]
\centering
\includegraphics[keepaspectratio=true,width=\textwidth, height=0.38\textheight]{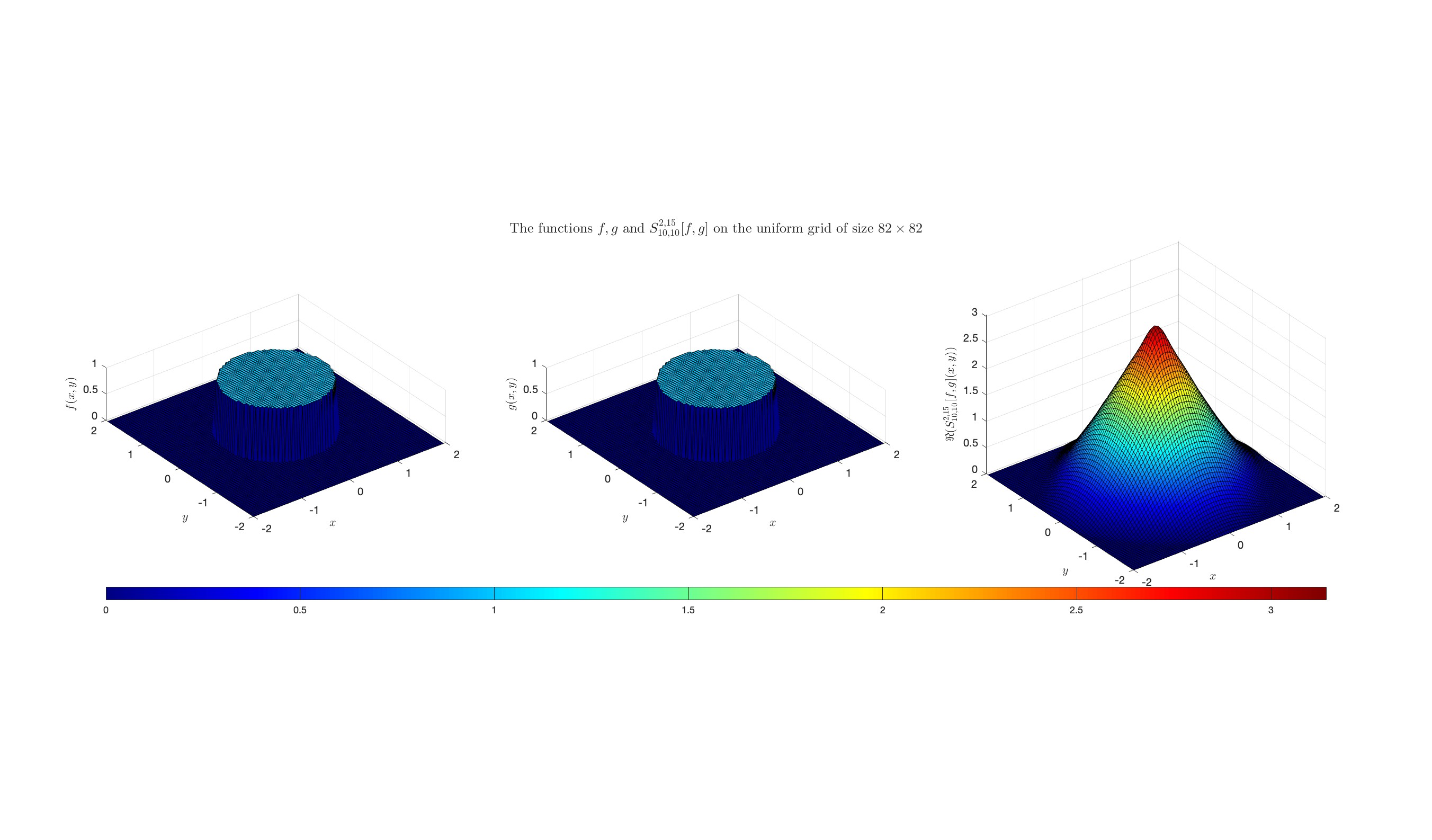}
\caption{The 3D surface plot of $f$, $g$, and $S_{10,10}^{3,15}[f,g]$ on the uniform grid of size $82\times 82$ in $\Omega_3$.  It is worth noting that $K[10,10]=15$.}
\label{fig:SurfE1E1}
\end{figure}
\begin{figure}[H]
\centering
\includegraphics[keepaspectratio=true,width=\textwidth, height=0.3\textheight]{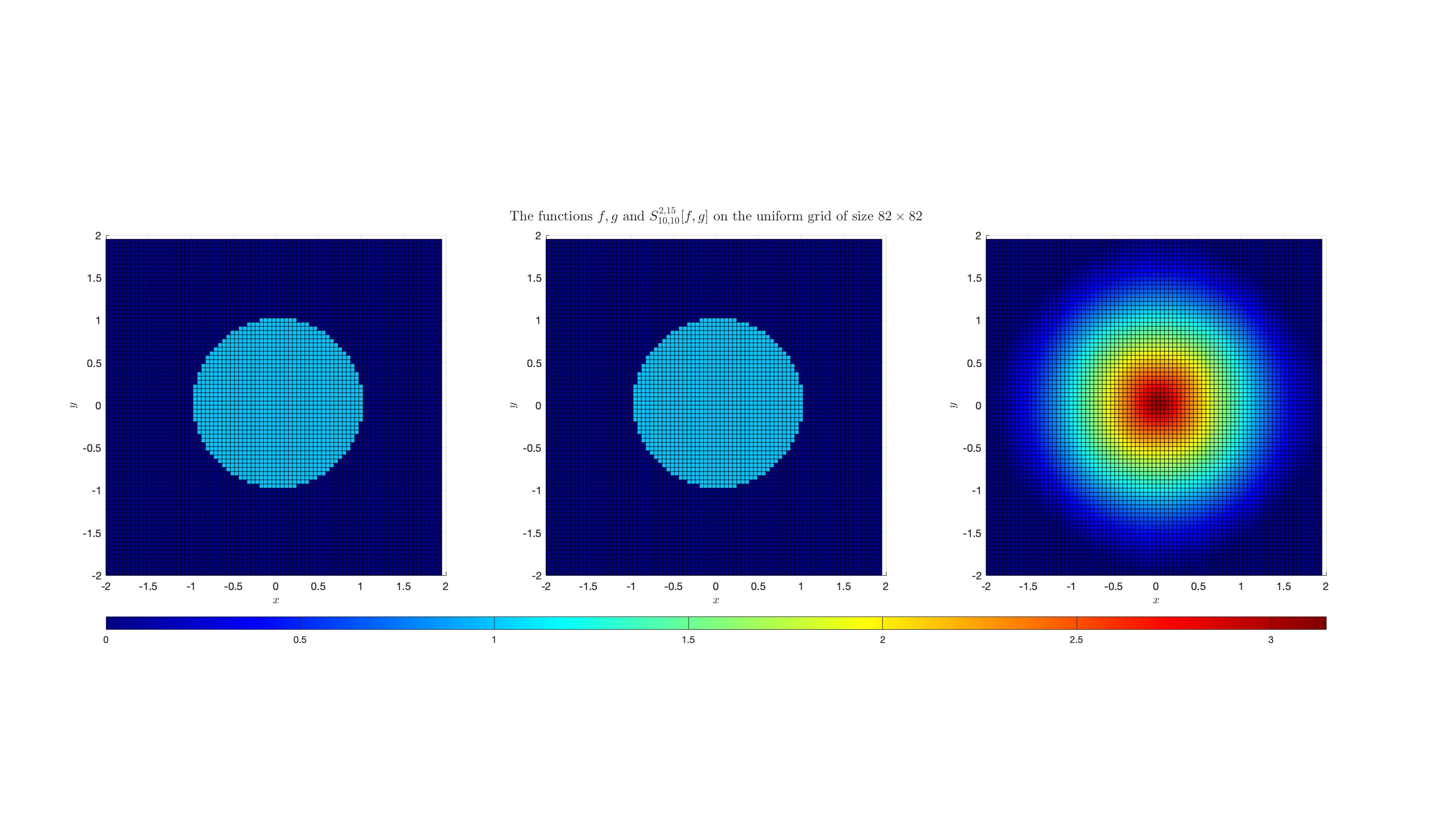}
\caption{The contour plot of $f$, $g$, and $S_{10,10}^{3,15}[f,g]$  on the uniform grid of size $82\times 82$ in $\Omega_3$. }
\label{fig:ContE1E1}
\end{figure}
\end{example}
\newpage
\begin{example}
Suppose that $f:=\chi_{\mathbb{B}_1}$ and $g:=\chi_{\mathbb{B}_2}$.
\begin{figure}[H]
\centering
\includegraphics[keepaspectratio=true,width=\textwidth, height=0.22\textheight]{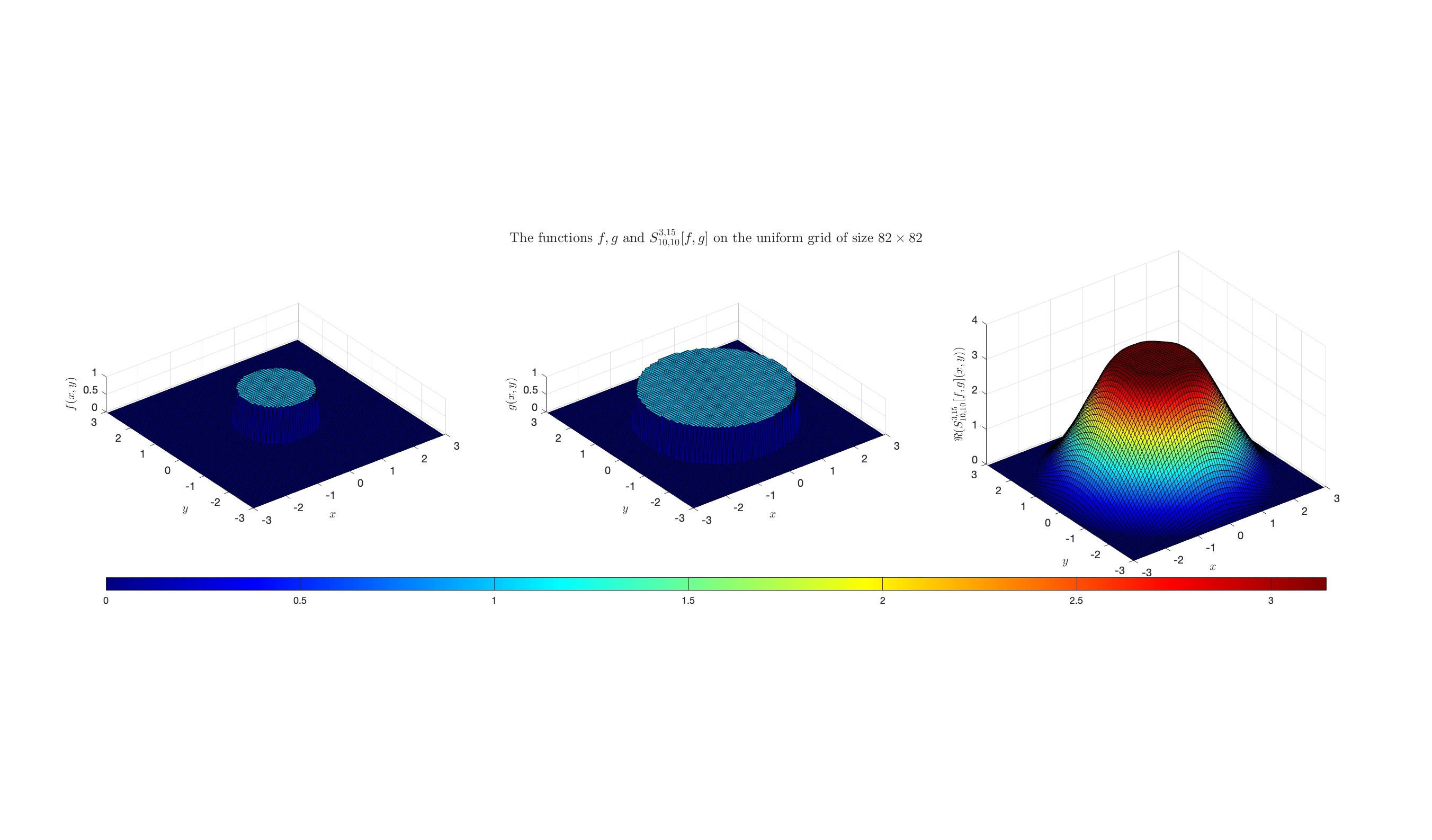}
\caption{The 3D surface plots of $f$, $g$, and $S_{10,10}^{3,15}[f,g]$ on the grid of size $82\times 82$.}
\label{fig:SurfE1E2}
\end{figure}
\begin{figure}[H]
\centering
\includegraphics[keepaspectratio=true,width=\textwidth, height=0.22\textheight]{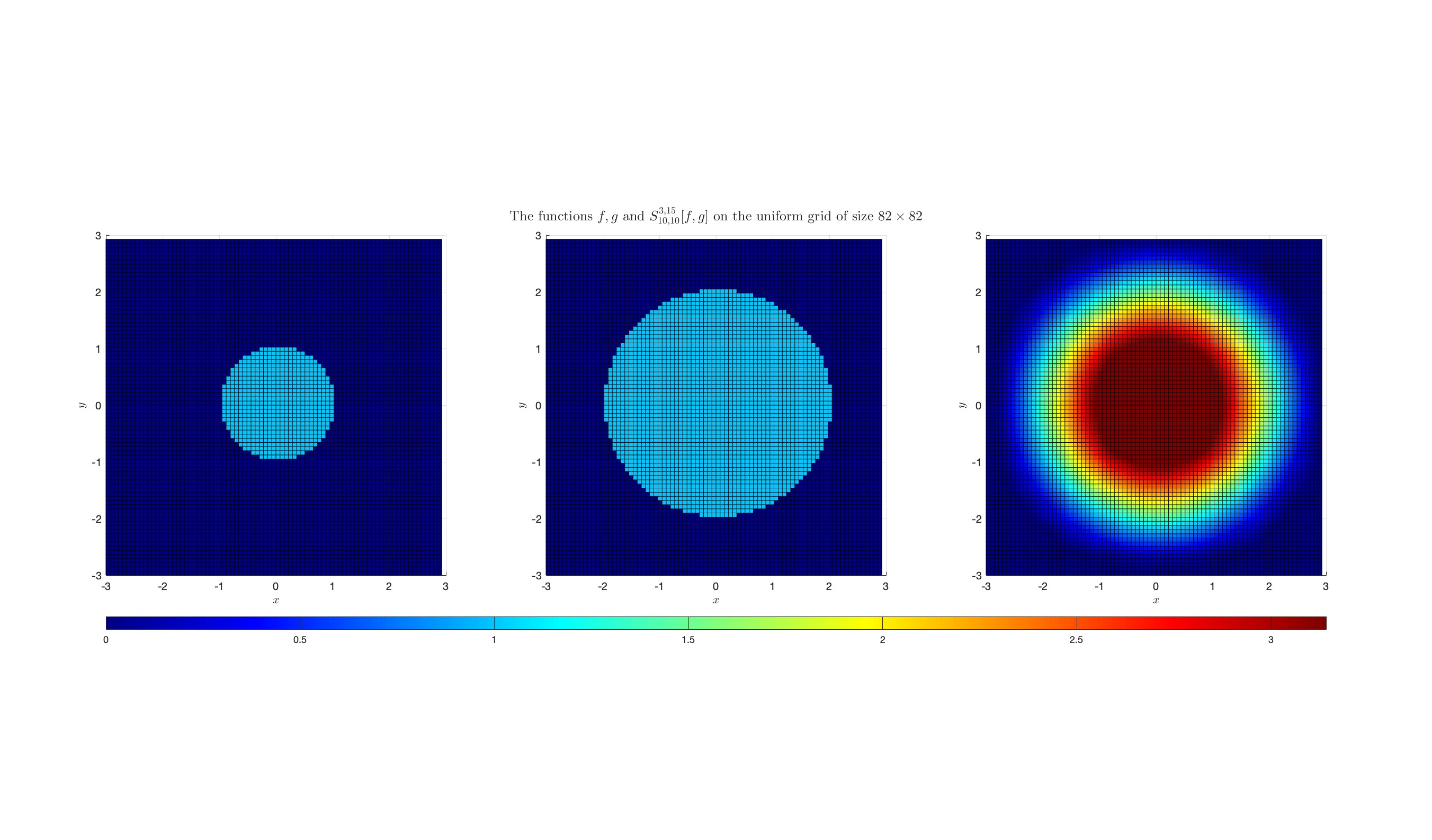}
\caption{The contour plots of $f$, $g$, and $S_{10,10}^{3,15}[f,g]$  on the grid of size $82\times 82$. }
\label{fig:ContE1E2}
\end{figure}
\end{example}
{\bf Concluding remarks.}
Discrete numerical methods using finite Fourier transform for approximating Fourier-Bessel coefficients (FFBT) and Fourier-Bessel partial sums (iFFBT) of functions on disks are introduced.  The approach taken here combines the best attributes of DFT/FFT for computing convolutions and the steerability of bandlimited Fourier-Bessel expansions. While bandlimiting in the polar basis loses the exactness of DFT/FFT convolution property, and finite sampling of bandlimited Fourier-Bessel series leads to a loss of exact steerability, both properties are recovered in the asymptotic limit of large band limit and sample size.

For functions which are smoothly supported in disks (absolutely/approximately) the numerical algorithms iFFBT gives very good convergence rate for the approximation in a few steps.  This is not so surprising as it has been guaranteed by the convergence Theorem \ref{SMNKer}.   For functions which are almost everywhere smooth, the iFFBT numerical algorithm shows some type of almost everywhere convergence after more steps of computations.  This is because of canonical Gibbs phenomenon originated from structure of polar harmonics near boundary of the disk due to discontinuities which occurs by zero-padding.   In the case,  that the ultimate goal is approximation of the support of the function, then the numerical algorithm for iFFBT is efficient yet.  We then introduced the numerical algorithm iFFBT,  for approximating convolutions supported on disks.

{\bf Acknowledgments.}
This research is supported by NUS Startup grants A-0009059-02-00 and A-0009059-03-00 and AME Programmatic Fund Project MARIO A-0008449-01-00.  The authors 
gratefully acknowledge the supporting agencies. The findings and opinions expressed here are only those of the authors, and not of the funding agencies. Thanks to Joël Bensoam,  and Ye Jikai for useful discussions related to the subject.

\appendix
\section{}\label{Apx}
This section contains some technical results, used in proofs of the main results. 
\subsection{Finite Fourier coefficients}\label{Apx1}
We here discuss a numerical scheme for Fourier coefficients. 

For $u\in L^1[-1,1] $ and $k\in\mathbb{Z}$, the Fourier coefficient  $\widehat{u}[k]$ is given by 
\begin{equation}\label{uk}
\widehat{u}[k]:=\frac{1}{2}\int_{-1}^1 u(x)e^{-\pi\ii kx}\D x.
\end{equation}
Assume $L\in\mathbb{N}$.  Then the finite Fourier coefficient of $u$ at $k\in\mathbb{Z}$ on the uniform grid of size $L$ is  
\begin{equation}\label{FFC}
\widehat{u}[k;L]:=\frac{1}{L}\sum_{i=1}^Lu(x_i)e^{-\pi\ii kx_i},
\end{equation}
where $x_i=-1+(i-1)\delta$ with $\delta:=\frac{2}{L}$.
In terms of classical integral calculus,  $\widehat{u}[k;L]$ is a left uniform Riemann sum of the integral $\widehat{u}[k]$ given by (\ref{uk}).  Next Lemma shows that the finite Fourier coefficient (\ref{FFC}) gives exact value of Fourier coefficient (\ref{uk}) for trigonometric polynomials. \footnote{For $K\in\mathbb{N}$,  we denote 
$\mathcal{T}_K[-1,1]:=\left\{p:[-1,1]\to\mathbb{C}:p(x)=\sum_{\ell=-K}^K\widehat{p}[\ell]e^{\pi\ii \ell x}\ \forall x\in[-1,1]\right\}$. }
\begin{lemma}\label{TrigL}
Let $K\in\mathbb{N}$ and $p\in\mathcal{T}_K[-1,1]$ be a trigonometric polynomial of degree $K$.  Then  
\begin{equation*}
\widehat{p}[k]=\widehat{p}[k;2K+1],\hspace{1.5cm}\ {\rm for}\ |k|\le K.
\end{equation*}
\end{lemma}
\begin{proof}
Let $L:=2K+1$ and $|k|\le K$.  Since $p\in\mathcal{T}_K[-1,1]$,  for every $1\le i\le L$, we get 
\[
p(x_i)=\sum_{\ell=-K}^K\widehat{p}[\ell]e^{\pi\ii \ell x_i}=\sum_{\ell=-K}^K\widehat{p}[\ell]e^{-\pi\ii \ell}e^{2\pi\ii \ell(i-1)/L}=\sum_{\ell=-K}^K\widehat{p}[\ell](-1)^{-\ell}e^{2\pi\ii \ell(i-1)/L},
\]
Assume $-K\le \ell\le K$.  Then $|\ell-k|\le 2K<L$.  Hence, 
$
\sum_{i=1}^Le^{2\pi\ii(\ell-k)(i-1)/L}=L\delta_{\ell,k}.
$
So
\begin{align*}
\widehat{p}[k;L]&=\frac{1}{L}\sum_{i=1}^Lp(x_i)(-1)^{k}e^{-2\pi\ii k(i-1)/L}
=\frac{1}{L}\sum_{i=1}^L\left(\sum_{\ell=-K}^K\widehat{p}[\ell](-1)^{-\ell} e^{2\pi\ii \ell(i-1)/L}\right)(-1)^{k}e^{-2\pi\ii k(i-1)/L}
\\&=\sum_{\ell=-K}^K\widehat{p}[\ell](-1)^{k-\ell}\left(\frac{1}{L}\sum_{i=1}^Le^{2\pi\ii(\ell-k)(i-1)/L}\right)
=\sum_{\ell=-K}^K\widehat{p}[\ell](-1)^{k-\ell} \delta_{\ell,k}=\widehat{p}[k].
\end{align*}
\end{proof}
We then conclude the following bound of the absolute error for approximation of $\widehat{u}[k]$ using $\widehat{u}[k,L]$.
\begin{proposition}\label{mainP}
{\it Let $u\in\mathcal{C}^1[-1,1]$ with $u(-1)=u(1)$,  and $K\in\mathbb{N}$. 
Then 
\[
\Big|\widehat{u}[k]-\widehat{u}[k;2K+1]\Big|\le\frac{12\|u'\|_\infty}{\pi K},\hspace{1.5cm}\ {\rm for}\ |k|\le K.
\]
}\end{proposition}
\begin{proof}
Let $v(t):=u(\pi^{-1}t)$, for $t\in[-\pi,\pi]$.  Then $v\in\mathcal{C}^1[-\pi,\pi]$ with $v(-\pi)=v(\pi)$ and $\|u'\|_\infty=\pi\|v'\|_\infty$.  Invoking Theorem 1.3 of \cite{Jack},  there exists a trigonometric polynomial $q$ of degree $K$ such that 
\begin{equation}\label{vq}
\|v-q\|_\infty\le\frac{6\|v'\|_\infty}{K}.
\end{equation}
Suppose $p(x):=q(\pi x)$ for $x\in[-1,1]$. Then $p\in\mathcal{T}_K[-1,1]$ and  using (\ref{vq}) we obtain 
\begin{equation}\label{JT}
\|u-p\|_\infty=\|v-q\|_\infty\le\frac{6\|u'\|_\infty}{\pi K}.
\end{equation}
Assume that $|k|\le K$ is given.  So,  using Lemma \ref{TrigL} and (\ref{JT}), we achieve 
\begin{align*}
\left|\widehat{u}[k]-\widehat{u}[k;L]\right|
&\le|\widehat{u}[k]-\widehat{p}[k]|+\left|\widehat{p}[k;L]-\widehat{u}[k;L]\right|
\le2\|u-p\|_\infty\le\frac{12\|u'\|_\infty}{\pi K}.
\end{align*}
\end{proof}
\begin{remark}
Let $u\in\mathcal{C}^1[-1,1]$ with $u(-1)=u(1)$ and $L\in\mathbb{N}$. Then Proposition \ref{mainP} guarantees that $\widehat{u}[k;L]$ is an accurate approximation of $\widehat{u}[k]$ if $|k|\le\lceil\frac{L-1}{2}\rceil$.
\end{remark}
For a function $U\in L^1(\Omega) $ and $\mathbf{k}:=(k_1,k_2)^T\in\mathbb{Z}$, the Fourier coefficient  $\widehat{U}[\mathbf{k}]$ is given by 
\begin{equation}\label{Uk}
\widehat{U}[\mathbf{k}]:=\frac{1}{4}\int_{\Omega}U(x,y)e^{-\pi\ii(k_1x+k_2y)}\D x\D y.
\end{equation}
If $L\in\mathbb{N}$ then the finite Fourier coefficient $\widehat{U}[\mathbf{k};L]$ is given by 
\begin{equation}\label{FFC2}
\widehat{U}[\mathbf{k};L]:=\frac{1}{L^2}\sum_{i=1}^L\sum_{j=1}^LU(x_i,x_j)e^{-\pi\ii (k_1x_i+k_2x_j)},
\end{equation}
where $x_i=-1+(i-1)\delta$ with $\delta:=\frac{2}{L}$.

Next we prove the following absolute error bound for approximation of $\widehat{U}[\mathbf{k}]$ using $\widehat{U}[\mathbf{k};L]$.

\begin{proposition}\label{mainT}
{\it Let  $U\in\mathcal{C}^1(\Omega)$ be a periodic function,  $\mathbf{k}:=(k_1,k_2)^T\in\mathbb{Z}^2$ and $K\in\mathbb{N}$.  Then 
\[
\left|\widehat{U}[\mathbf{k}]-\widehat{U}[\mathbf{k};2K+1]\right|\le\frac{24\|\nabla U\|_\infty}{\pi K},\hspace{1.5cm}\ {\rm if}\ \|\mathbf{k}\|_\infty\le K.
\]
}\end{proposition}
\begin{proof}
Let $U_x(y):=U(x,y)$ and $u_{k_2}(x):=\widehat{U_x}[k_2]$ for $(x,y)\in\Omega$.  Then 
\begin{align*}
&\widehat{u_{k_2}}[k_1]
=\frac{1}{2}\int_{-1}^1u_{k_2}(x)e^{-\pi\ii k_1x}\D x=\frac{1}{2}\int_{-1}^1\widehat{U_x}[k_2]e^{-\pi\ii k_1x}\D x
\\&=\frac{1}{4}\int_{-1}^1\left(\int_{-1}^1U_x(y)e^{-\pi\ii k_2 y}\D y\right)e^{-\pi\ii k_1x}\D x
=\frac{1}{4}\int_{-1}^1\int_{-1}^1U(x,y)e^{-\pi\ii (k_1x+k_2 y)}\D x\D y=\widehat{U}[\mathbf{k}].
\end{align*}
Using Leibniz integral rule,  $u_{k_2}'(x)=\widehat{(\partial_1U)_x}[k_2]$,  
implying $|u_{k_2}'(x)|\le\|\partial_1U\|_\infty$. So $\|u_{k_2}'\|_\infty\le\|\partial_1U\|_\infty$.  
Suppose $L:=2K+1$ and $\|\mathbf{k}\|_\infty\le K$.  Then, using Proposition \ref{mainP},  we achieve 
\begin{align*}
\left|\widehat{U}[\mathbf{k}]-\widehat{U}[\mathbf{k};L]\right|
&\le\left|\widehat{u_{k_2}}[k_1]-\widehat{u_{k_2}}[k_1;L]\right|+\left|\widehat{u_{k_2}}[k_1;L]-\widehat{U}[\mathbf{k};L]\right|
\\&\le\frac{12\left\|u_{k_2}'\right\|_\infty}{\pi K}+\frac{1}{L}\sum_{i=1}^L\left|\widehat{U_{x_i}}[k_2]-\widehat{U_{x_i}}[k_2;L]\right|
\le\frac{12\|\partial_1U\|_\infty}{\pi K}+\frac{1}{L}\sum_{i=1}^L\frac{12\|U_{x_i}'\|_\infty}{\pi K}\\&\le\frac{12\|\partial_1U\|_\infty}{\pi K}+\frac{1}{L}\sum_{i=1}^L\frac{12\|\partial_2 U\|_\infty}{\pi K}
\le\frac{12\|\partial_1U\|_\infty}{\pi K}+\frac{12\|\partial_2 U\|_\infty}{\pi K}\le \frac{24\|\nabla U\|_\infty}{\pi K}.
\end{align*}
\end{proof}
\begin{remark}
Let $U\in\mathcal{C}^1(\Omega)$ be a periodic function and $L\in\mathbb{N}$. Then Proposition \ref{mainT} guarantees that $\widehat{U}[\mathbf{k};L]$ is an accurate approximation of the Fourier coefficient $\widehat{U}[\mathbf{k}]$ if $\|\mathbf{k}\|_\infty\le\lceil\frac{L-1}{2}\rceil$.
\end{remark}
\subsection{Asymptotic/symmetry analysis of $\mathbf{c}(\mathbf{k};m,n)$}\label{Apx2}
\begin{lemma}\label{cakmnSym}
{\it Let $m\in\mathbb{Z}$, and $\mathbf{k}=(k_1,k_2)\in\mathbb{Z}^2$. Then
\begin{enumerate}
\item $\mathbf{c}(\mathbf{k};-m,n)=\overline{\mathbf{c}(\mathbf{k};m,n)}$ if $m\ge 0$.
\item $\mathbf{c}(-\mathbf{k};m,n)=(-1)^m\mathbf{c}(\mathbf{k};m,n)$ if $m\ge 0$.
\end{enumerate}
}\end{lemma}
\begin{proof}
(1) Since $-m\le 0$,  using (\ref{cakmn}), we achieve 
\begin{align*}
\mathbf{c}(\mathbf{k};-m,n)
&=(-1)^{-m}\sqrt{\pi}(-1)^{n}\ii^{-m}z_{-m,n}\frac{J_{-m}(\pi\|\mathbf{k}\|_2)e^{\ii m\Phi(\mathbf{k})}}{2(\pi^2\|\mathbf{k}\|_2^2-z_{-m,n}^2)}
\\&=\sqrt{\pi}(-1)^{n}\ii^{-m}z_{m,n}\frac{J_{m}(\pi\|\mathbf{k}\|_2)e^{\ii m\Phi(\mathbf{k})}}{2(\pi^2\|\mathbf{k}\|^2_2-z_{m,n}^2)}=\overline{\mathbf{c}(\mathbf{k};m,n)}.
\end{align*}
(2) We have 
\begin{align*}
\mathbf{c}(-\mathbf{k};m,n)
&=\sqrt{\pi}(-1)^{n}\ii^{m}z_{m,n}\frac{J_{m}(\pi\|-\mathbf{k}\|_2)e^{-\ii m\Phi(-\mathbf{k})}}{2(\pi^2\|-\mathbf{k}\|_2^2-z_{m,n}^2)}
\\&=\sqrt{\pi}(-1)^{n}\ii^{m}z_{m,n}\frac{J_{m}(\pi\|\mathbf{k}\|_2)(-1)^me^{-\ii m\Phi(\mathbf{k})}}{2(\pi^2\|\mathbf{k}\|_2^2-z_{m,n}^2)}=(-1)^m\mathbf{c}(\mathbf{k};m,n).
\end{align*}
\end{proof}
For every $(m,n)\in\mathbb{Z}\times\mathbb{N}$, let $K_{m,n}:=\lceil\frac{z_{m,n}}{\pi}\rceil$.
\begin{lemma}\label{OmainKmn}
Let $(m,n)\in\mathbb{Z}\times\mathbb{N}$ and $t\ge K_{m,n}$. 
Then 
\[
\frac{1}{\pi^2t^2-z_{m,n}^2}\le\frac{K_{m,n}^2}{\pi^2K_{m,n}^2-z_{m,n}^2}\frac{1}{t^2}.
\]
\end{lemma}
\begin{proof}
Let $t\ge K_{m,n}$.  Then $\pi^2t^2\ge \pi^2K_{m,n}^2$.
Hence,  $\pi^2t^2\ge \pi^2K_{m,n}^2>z_{m,n}^2$.  Therefore, we achieve  \footnote{If $x\ge A>B>0$ then $\frac{1}{x-B}\le \frac{A}{A-B}\frac{1}{x}$.}
\begin{equation*}
\frac{1}{\pi^2t^2-z_{m,n}^2}\le\frac{\pi^2K_{m,n}^2}{\pi^2K_{m,n}^2-z_{m,n}^2}\frac{1}{\pi^2t^2}=\frac{K_{m,n}^2}{\pi^2K_{m,n}^2-z_{m,n}^2}\frac{1}{t^2}.
\end{equation*}
\end{proof}
\begin{lemma}\label{SumJmk}
{\it Let $(m,n)\in\mathbb{Z}\times\mathbb{N}$.  Then 
\[
\sum_{\mathbf{k}\in\mathbb{Z}^2}|\mathbf{c}(\mathbf{k};m,n)|<\infty.
\]
}\end{lemma}
\begin{proof}
Suppose that $\|\mathbf{k}\|_2>K_{m,n}$.  Then Lemma \ref{OmainKmn},  for $t:=\|\mathbf{k}\|_2$,  concludes that 
\begin{equation}\label{Jm2mn}
\frac{|J_{m}(\pi\|\mathbf{k}\|_2)|}{|\pi^2\|\mathbf{k}\|^2_2-z_{m,n}^2|}=\frac{|J_{m}(\pi\|\mathbf{k}\|_2)|}{\pi^2\|\mathbf{k}\|^2_2-z_{m,n}^2}
\le\frac{K_{m,n}^2}{\pi^2K_{m,n}^2-z_{m,n}^2}\frac{|J_{m}(\pi\|\mathbf{k}\|_2)|}{\|\mathbf{k}\|_2^2}.
\end{equation}
Assume $\alpha,D_m>0$ such that $|J_m(x)|\le \alpha|x|^{-1/2}$ as $|x|>D_m$. Let $q_{m,n}:=\max\{K_{m,n},D_m\}$. If $\|\mathbf{k}\|_\infty>q_{m,n}$ then $\|\mathbf{k}\|_2\ge\|\mathbf{k}\|_\infty>q_{m,n}$. Hence,  $\|\mathbf{k}\|_2> D_m$ which gives
\begin{equation}\label{Jm2}
\frac{|J_{m}(\pi\|\mathbf{k}\|_2)|}{\|\mathbf{k}\|_2^2}\le\frac{\alpha\sqrt{\pi}}{\|\mathbf{k}\|_2^{5/2}}.
\end{equation}
Then using (\ref{Jm2mn}),  (\ref{Jm2}), and Lemma 1 of \cite[Page 96]{M.Sugiu} we achieve 
\begin{align*}
\sum_{\|\mathbf{k}\|_\infty>q_{m,n}}\frac{|J_{m}(\pi\|\mathbf{k}\|_2)|}{\|\mathbf{k}\|_2^2}&\le\alpha\sqrt{\pi}\sum_{\|\mathbf{k}\|_\infty>q_{m,n}}\frac{1}{\|\mathbf{k}\|_2^{5/2}}\le\alpha\sqrt{\pi}\sum_{\mathbf{k}\in\mathbb{L}^2}\frac{1}{\|\mathbf{k}\|_2^{5/2}}= \alpha\sqrt{\pi}\zeta_2(5/4),
\end{align*}
where $\zeta_2$ is Eisenstein's zeta function.  This implies that 
\begin{align*}
\sum_{\|\mathbf{k}\|_\infty>q_{m,n}}|\mathbf{c}(\mathbf{k};m,n)|
&=\frac{\sqrt{\pi}z_{m,n}}{2}\sum_{\|\mathbf{k}\|_\infty>q_{m,n}}\frac{|J_{m}(\pi\|\mathbf{k}\|_2)|}{|\pi^2\|\mathbf{k}\|^2_2-z_{m,n}^2|}
\\&\le\frac{\sqrt{\pi}z_{m,n}}{2}\frac{K_{m,n}^2}{\pi^2K_{m,n}^2-z_{m,n}^2}\sum_{\|\mathbf{k}\|_\infty>q_{m,n}}\frac{|J_{m}(\pi\|\mathbf{k}\|_2)|}{\|\mathbf{k}\|_2^2}
\le\frac{\pi\alpha z_{m,n}K_{m,n}^2}{2(\pi^2K_{m,n}^2-z_{m,n}^2)}\zeta_2(5/4).
\end{align*}
Therefore, we obtain 
\begin{align*}
\sum_{\mathbf{k}\in\mathbb{Z}^2}|\mathbf{c}(\mathbf{k};m,n)|
&=\sum_{\|\mathbf{k}\|_\infty\le q_{m,n}}|\mathbf{c}(\mathbf{k};m,n)|+\sum_{\|\mathbf{k}\|_\infty>q_{m,n}}|\mathbf{c}(\mathbf{k};m,n)|
\\&\le\sum_{\|\mathbf{k}\|_2\le q_{m,n}}|\mathbf{c}(\mathbf{k};m,n)|+\frac{\pi\alpha z_{m,n}K_{m,n}^2}{2(\pi^2K_{m,n}^2-z_{m,n}^2)}\zeta_2(5/4)<\infty.
\end{align*}
\end{proof}

\bibliographystyle{amsplain}

\end{document}